\documentclass[12pt]{article}
\usepackage{amsfonts,amssymb,amsmath,fullpage,pst-node,rotating}

\newcommand{\R}{{\mathbb R}}

\newcommand{\C}{{\mathbb C}}

\newcommand{\bx}{{\bf x}}

\newcommand{\cB}{{\cal B}}

\newcommand{\rank}{{\rm rank}}

\newcommand{\cX}{{\cal X}}

\newcommand{\mbe}{\mbox{\boldmath${\eta}$}}

\newcommand{\mbi}{\mbox{\boldmath${\infty}$}}
\newcommand{\mbel}{{\boldmath${\eta}$}}

\newtheorem{definition}{Definition}
\newtheorem{theorem}{Theorem}
\newtheorem{lemma}{Lemma}

\newtheorem{remark}{Remark}
\newtheorem{proposition}{Proposition}
\newcommand{\proof}{\noindent {\bf Proof: }}

\newcommand{\qed}{\hfill{\bf QED}
\vspace{5mm}}

\begin{document}

\title{Asymptotic stability of heteroclinic cycles of type Y}

\author{Olga Podvigina\\
Institute of Earthquake Prediction Theory and\\
Mathematical Geophysics of the Russian Academy of Sciences,\\
84/32 Profsoyuznaya St, 117997 Moscow, Russian Federation
}

\maketitle

\begin{abstract}
We investigate stability of a new class of
heteroclinic cycles that we call {\it heteroclinic cycles of type Y}.
The cycles can be regarded as a generalisation of heteroclinic
cycles of type Z introduced in [Podvigina, Nonlinearity 25, 2012].
The type Y cycles differ from the cycles of type Z in the following:
The trajectories comprising a cycle of type Y belong to flow-invariant
subspaces that can be of different dimensions.
Unlike in the most studies of the stability of heteroclinic cycles,
we do not require that the eigenvalues of the linearisations of the
dynamical system near the equilibria are distinct.
Instead of the common assumption that the cycles are robust,
we prescribe flow-invariance of certain subspaces.
Similarly to type Z cycles, asymptotic stability and
fragmentary asymptotic stability of type Y cycles is determined by the eigenvalues
and eigenvectors of transition matrices. The matrices are products of
basic transition matrices that depend on the eigenvalues of linearisations and
the dimensions of the contracting subspaces.
\end{abstract}

\noindent

\section{Introduction}\label{sec_intro}

A heteroclinic cycle is an invariant set of a smooth dynamical system
\begin{equation}\label{eq_ode}
\dot{\bf x}=f({\bf x}),\quad f:\R^n\to\R^n,
\end{equation}
comprised of nodes and heteroclinic connections. The cycles often emerge in
the system related to population dynamics, chemistry, fluid dynamics,
neuroscience and many others \cite{ash24,hs98,may75,pj98,rab06}.
In geophysics, heteroclinic cycles are believed to be responsible for magnetic
field reversal and long-time climate variations \cite{ch03,cro03,cru12,pf10,op20}.

The existence, stability and bifurcations of heteroclinic
cycles have been extensively studied in the literature since 80s.
In the majority of these studies it was assumed that the nodes, involved in the
cycle, are equilibria and the connections are heteroclinic trajectories.
Following this assumption, in the present paper we study asymptotic stability
of a heteroclinic cycle $\cX$ comprised of equilibria, $\xi_j$,
and connecting trajectories, $\kappa_j$,
\begin{equation}\label{hetc}
\cX=\{\ \xi_1,...,\xi_m\ ;\ \kappa_1,...,\kappa_m\ \},\
\kappa_j:\xi_j\to\xi_{j+1},\ j=1,...,m-1,\ \kappa_m:\xi_m\to\xi_1.
\end{equation}

A set that attracts all trajectories in its small neighbourhood is called
asymptotically stable. Heteroclinic cycles that are not asymptotically stable
can attract a positive measure set from its small neighbourhood.
In \cite{op12} such cycles were
called {\it fragmentarily asymptotically stable}.
Asymptotic stability or fragmentary asymptotic stability of heteroclinic
cycles was considered in a number of papers. Conditions for stability of
so-called simple heteroclinic cycles in $\R^n$ with $n=3$ or 4 were proven,
e.g., in \cite{bra94,ks94,km95a,km04,pa11}, while more general cycles in
four-dimensional systems were considered in \cite{cas25,clp}.
Stability of particular cycles in dimensions $n>4$ was investigated in
\cite{ac10,bic19,pos22a,pos22b,cas22},
while cycles in dynamical systems of arbitrary dimension was studied in
\cite{gar,op12}.

In all these studies heteroclinic cycles were assumed to be robust.
Generically heteroclinic cycles are structurally
unstable, because an arbitrary small perturbation of $f$ in (\ref{eq_ode})
breaks a connection between two saddle steady states. However, the connections
can be structurally stable (or robust) if the dynamical system has a non-trivial
symmetry group and only symmetric perturbations are considered,
or if the system is constrained to preserve certain invariant subspaces
\cite{km04,fi96,gukhol,Kru97}.
Given that all connections
$\kappa_j:\xi_j\to\xi_{j+1}$ belong to invariant subspaces $P_j$,
where $\xi_{j+1}$ is stable in this subspace, the cycle persists with respect to
perturbations, preserving the subspaces.
Being robust, heteroclinic cycles can be observed in simulations or
experiments, for example, in rotating
convection between two plates and for turbulent flows in a boundary layer.
Another typical assumption in the studies of stability is that the
eigenvalues of $df(\xi_j)$ are distinct. In particular, simple heteroclinic
cycles satisfy this assumption.

In the present paper we introduce a class of heteroclinic cycles in $\R^n$,
that we call {\it type Y} heteroclinic cycles. Almost all heteroclinic
cycles cited above, including quasi-simple studied in \cite{gar} and cycles in
pluridimensions \cite{cas25},
belong to this class. It can be regarded as a generalisation of the
type Z cycles \cite{op12}. The following has been changed:

Type Y heteroclinic cycle, unlike type Z one, is not required to be robust.
Instead, we assume invariance of certain subspaces. In particular, the
invariance of $P_j\supset\kappa_j$ is assumed. The derivation of
the conditions for stability relies on the structure of local and
global maps that approximate a trajectory in the vicinity of a cycle.
For maps, the essential thing is the invariance of subspaces and not the
nature of the invariance.

We do not require the eigenvalues of $df(\xi_j)$ to be distinct. The invariance
of subspaces determines whether $df(\xi_j)$ may possess complex eigenvalues
and/or Jordan cells, therefore providing the expressions for local maps.
We do not assume that the dimensions of $P_j$ are equal,
as this was the case in \cite{op12}. The requirement is that the
dimension of the expanding subspace is one, namely that
$\dim P_j-\dim(P_j\cap P_{j-1})=1$.

The main result of this paper is that, despite the stated above differences,
the conditions for asymptotic stability expressed in the terms of eigenvalues
and eigenvectors of transition matrices for type Y cycles are the same as
the ones for type Z cycles proven in \cite{op12}.
A cycle is asymptotically stable, or
fragmentarily asymptotically stable, whenever certain inequalities on
eigenvalues and
eigenvectors of the transition matrices associated with the cycle are
satisfied. The transition matrices
are products of basic transition matrices. The basic matrices depend
on the dimension the contracting subspace equal to
$\dim P_{j-1}-\dim(P_j\cap P_{j-1})$ and on the eigenvalues of $df(\xi_j)$.

In section \ref{sec_defs} we recall definitions of asymptotic stability,
heteroclinic cycles, local and global maps, and define type Y heteroclinic
cycles. In sections \ref{stabcm} and \ref{smaps}
we derive necessary and sufficient conditions
for asymptotic stability and fragmentary asymptotic stability of the cycles.
In section \ref{exa} we give two examples of type Y
heteroclinic cycles in the generalised Lotka-Volterra system. One cycle
is asymptotically stable and the other one is
fragmentarily asymptotically stable. In the conclusion we briefly summarise
the results and outline directions for further studies. In appendix
A we prove a theorem about the rank of transition matrices, and in
appendix B two theorems providing
sufficient conditions for the existence of heteroclinic connections
comprising the cycles of section \ref{exa}.

\section{Definitions}\label{sec_defs}

\subsection{Stability}\label{sec_stab}

Given a set $X$ and a number $\varepsilon>0$, the
$\varepsilon$-neighbourhood of $X$ is
\begin{equation}\label{ep_nei}
B_{\varepsilon}(X)=\{{\bf x}\in \R^n:\ d({\bf x},X)<\varepsilon\}.
\end{equation}
Let $\Phi(\tau,{\bf x})$ be the trajectory
of system (\ref{eq_ode}) through ${\bf x}\in\R^n$.
For $X$, a compact invariant set of (\ref{eq_ode}), its
$\delta$-local basin of attraction is
\begin{equation}\label{del_bas}
\cB_{\delta}(X)=\{{\bf x}\in\R^n:\ d(\Phi(\tau,{\bf x}),X)<\delta\hbox{ for any }
\tau\ge0\hbox{ and }\lim_{\tau\to\infty}d(\Phi(\tau,{\bf x}),X)=0\}.
\end{equation}

\begin{definition}\label{def1}
A compact invariant set $X$ is called \underline{asymptotically stable}, if for any
$\delta>0$ there exists an $\varepsilon>0$ such that
$$B_{\varepsilon}(X)\subset\cB_{\delta}(X).$$
\end{definition}

\begin{definition}\label{def2}\cite{op12}
A compact invariant set $X$ \underline{fragmentarily asymptotically
stable}, if for any $\delta>0$
$$\mu(\cB_{\delta}(X))>0.$$
(Here $\mu$ is the Lebesgue measure of a set in $\R^n$.)
\end{definition}

\begin{definition}\label{def3}
A set $X$ is called \underline{completely unstable}, if there exists
$\delta>0$ such that $\mu(\cB_{\delta}(X))=0$.
\end{definition}

\subsection{Heteroclinic cycles}\label{defhet}

Let $\xi_1,\ldots,\xi_m\in\R^n$ be hyperbolic equilibria of (\ref{eq_ode})
and $\kappa_j:\xi_j\to\xi_{j+1}$, $j=1,\ldots,m$,
$\xi_{m+1}=\xi_1$, trajectories from $\xi_j$ to $\xi_{j+1}$.
\begin{definition}\label{defhc}
A heteroclinic cycle $\cX\in\R^n$ is a union of equilibria $\{\xi_1,\ldots,\xi_m\}$
and connecting trajectories $\{\kappa_1,\ldots,\kappa_m\}$.
\end{definition}

\begin{definition}\label{def72}
A heteroclinic cycle $\cX$ is of \underline{type Y} if
\begin{itemize}
\item[I]
For any $j$, $1\le j\le m$, the connection $\kappa_j$ belongs to $P_j$, an
invariant subspace of (\ref{eq_ode}), such that
$\dim P_j=\dim L_j+1$, where $L_j=P_{j-1}\cap P_j$.
\item[II] Let $\{v_1,...,v_K\}$ be the eigenvectors of $df(\xi_j)$ that
do not belong to $L_j$. For any set of indices $(k_1,...,k_d)$ the subspace
$L_j\oplus<v_{k_1},...,v_{k_d}>$ is an invariant subspace of (\ref{eq_ode}).
\end{itemize}
\end{definition}

\begin{remark}
For system (\ref{eq_ode}) considered in $\R^n$ the heteroclinic cycle of definition
\ref{defhc} is never asymptotically
stable according to definition \ref{def1} because the cycle does not contain
the whole unstable manifold of $\xi_j$ for any $j$.
To overcome this, in a $\Gamma$-equivariant system
a heteroclinic cycle is assumed to be (a connected
component of) the orbit under the action of the group of symmetries $\Gamma$
of the cycle in terms of definition \ref{defhc}.
In a non-symmetric context the system is typically considered
in $\R^n_+$, as this is the case, e.g., in population dynamics.
Any Cartesian plane or hyperplane is invariant due to the permanence of death.
The condition that the dimension of the unstable manifold is one implies
that the expanding eigenvector is a Cartesian basis vector. Hence, a half
of the unstable manifold goes to the negative direction and does not
belong to the phase space of the system.
\end{remark}

\subsection{Eigenvalues and eigenvectors}
\label{locstr}

Consider $\xi_j\in L_j=P_{j-1}\cap P_j$.
Following \cite{km95a,km95b,km04} for a type Y heteroclinic
cycle we divide eigenvalues of $df(\xi_j)$ into four classes:
\begin{itemize}
\item Eigenvalues with associated eigenvectors in $L_j$ are called {\it radial}.
\item Eigenvalues with associated
eigenvectors in $P_{j-1}\ominus L_j$ are called {\it contracting}.
\item Eigenvalues
with associated eigenvectors in $P_j\ominus L_j$ are called {\it expanding}.
\item Eigenvalues not belonging to any of the three above classes
are called {\it transverse}.
\end{itemize}

The condition $\dim P_j-\dim L_j=1$ implies that for any $j$ the
expanding subspace at $\xi_j$ is one-dimensional, while the dimensions
of other subspaces can be arbitrary.
Denote by $n_j^r$, $n_j^c$ and $n_j^t$ the number of radial,
contracting and transverse eigenvalues, respectively.
The radial eigenvalues and the associated eigenvectors near $\xi_j$ are denoted
by ${\bf r}_j=\{r_{j,l}\}$ and ${\bf v}_j^r=\{v_{j,l}^r\}$,
$1\le l\le n_j^r$,
the contracting ones
by ${\bf c}_j=\{c_{j,d}\}$ and ${\bf v}_j^c=\{v_{j,d}^c\}$,
$1\le d\le n_j^c$,
the expanding ones by $e_j$ and
$v_j^e$, and the transverse ones by ${\bf t}_j=\{t_{j,s}\}$ and
${\bf v}_j^t=\{v_{j,s}^t\}$, $1\le s\le n_j^t$, respectively.

Consider the bases in $\R^n$ comprised of the eigenvectors of $df(\xi_j)$.
The subspace $P_j$ is spanned by the radial and expanding eigenvectors
of $df(\xi_j)$, or by radial and contracting eigenvectors of $df(\xi_{j+1})$.
Denote by $P_j^{\perp}$ its complement spanned either by contracting and
transverse eigenvectors of $df(\xi_j)$, or by expanding and transverse
eigenvectors of $df(\xi_{j+1})$. The definition of type Y cycles implies that
the basis $\{v_{j+1}^e,{\bf v}_{j+1}^t\}$ in $P_j^{\perp}$ is a permutation of
the basis $\{v_j^c,{\bf v}_j^t\}$.
Hence, the matrix $A_j$ mapping components
of a vector in the basis $\{v_j^c,{\bf v}_j^t\}$ to components of the vector
in the basis $\{v_{j+1}^e,{\bf v}_{j+1}^t\}$ is a permutation matrix.

\medskip
Following \cite{ks94,km04,pa11,op12}, in order to examine stability
we construct a Poincar\'e map in the vicinity of the cycle.

\subsection{Collection of maps associated with a heteroclinic cycle}
\label{cycmap}

\begin{lemma}\label{lecol}
Consider a smooth mapping $g:\R^n\to\R^n$. Suppose that $\R^n=X\oplus Y$,
where the subspace $X$ is $g$-invariant, $Y$ is one-dimensional,
${\bf x}$ and $y$ are the coordinates in $X$ and $Y$, respectively
and $g_{X}$ and $g_{Y}$ denote the projections of $g$ on $X$ and $Y$.
Then for small $\bf x$ and $y$ the $Y$ component of $g$ satisfies
\begin{equation}\label{lemest}
g_{Y}({\bf x},y)=Cy^m(1+q({\bf x},y)),
\end{equation}
where $q({\bf x},y)<C_x|{\bf x}|+C_y|y|$, $m\ge1$ and generically $m=1$.
\end{lemma}

\proof
The smoothness of $g$ implies that for small ${\bf x}$ and $y$ we can write
$$g_{Y}({\bf x},y)=F_0({\bf x})+\sum_{1\le k<\infty}F_k({\bf x})y^k.$$
The invariance of $X$ implies that $F_0({\bf x})\equiv0$.
Setting $m$ to be the smallest $k$ such that $F_k({\bf x})$ is not
identically zero and using that $g$ is smooth we prove the lemma.
Generically, we have that $F_1({\bf x})\not\equiv0$
\qed

In subsection~\ref{locstr} we have defined radial, contracting,
expanding and transverse eigenvalues of the linearisation $df(\xi_j)$.
Denote $n_j=\dim P_j^{\perp}$. The definition of type Y cycles implies that
$n_j^c+n_j^t=n_{j+1}^t+1=n_j$ and
$n_j^r+1=n_{j+1}^r+n_{j+1}^c=\dim P_j=n-n_j$.
Let $(\tilde{\bf u},\tilde {\bf v},\tilde w,\tilde{\bf z})$ be local coordinates
near $\xi_j$ in the basis, where radial eigenvectors come the first
(the respective coordinates are $\tilde{\bf u}$), followed by the contracting
and the expanding eigenvectors, the transverse eigenvectors being the last.

Suppose $h$ is small. Denote by $B_{h}(\xi_j)$ the $h$-neighbourhood of $\xi_j$,
$$
B_{h}(\xi_j)=\{(\tilde{\bf u},\tilde {\bf v},\tilde w,\tilde{\bf z})\ :\
\max(|\tilde{\bf u}|,|\tilde {\bf v}|,|\tilde w|,|\tilde{\bf z}|)<h\}.
$$
The definition of type Y cycles implies that for each of contracting,
expanding and transverse coordinates the conditions of lemma \ref{lecol}
for $g$ being the r.h.s. of (\ref{eq_ode}) are satisfied.
Therefore, by lemma \ref{lecol}  near $\xi_j$ for each of these
coordinates the system can be approximated by $\dot y=Cy$. We have
\begin{equation}\label{lmap1}
\renewcommand{\arraystretch}{1.2}
\begin{array}{l}
\dot v_d=c_{j,d} v_d,\ 1\le d\le n^c_j\\
\dot w=e_j w\\
\dot z_s=t_{j,s}z_s,\ 1\le s\le n^t_j.
\end{array}
\end{equation}
The radial coordinates satisfy
\begin{equation}\label{lmap2}
\dot {\bf u}=A^{rad}{\bf u}.
\end{equation}
The matrix $A^{rad}$ may possess Jordan cells and/or complex eigenvalues.
We denote by $({\bf u},{\bf v},w,{\bf z})$ the scaled coordinates
$({\bf u},{\bf v},w,{\bf z})=
(\tilde {\bf u},\tilde {\bf v},\tilde w,\tilde {\bf z})/h$.

Consider a neighbourhood of a steady state $\xi_j$. Let $({\bf u}_0,{\bf v}_0)$ be
the point in $P_{j-1}$ where trajectory $\kappa_{j-1}$ intersects
with the sphere $|{\bf u}|^2+|{\bf v}|^2=1$, and $\bf q$ be local
coordinates in the hyperplane tangent to the sphere at the point
$({\bf u}_0,{\bf v}_0)$. Coordinates $({\bf u},{\bf v})$ of a point in the hyperplane
are related to coordinates $\bf q$ as follows:
\begin{equation}\label{relat}
\left(
\begin{array}{c}
{\bf u}\\
{\bf v}
\end{array}
\right)={\cal D}_j{\bf q}=
\left(
\begin{array}{c}
{\bf u}_0\\
{\bf v}_0
\end{array}
\right)+D_j{\bf q},
\end{equation}
where $D_j$ is an $(n-n_j)\times(n-n_j-1)$ matrix.

Near $\xi_j$ we define two crossections of the heteroclinic cycle. One,
denoted by $\tilde H^{(in)}_j$, is an $(n-1)$-dimensional hyperplane
intersecting connection $\kappa_{j-1}$ at the point $({\bf u}_0,{\bf v}_0,0,0)$;
coordinates in the hyperplane are $({\bf q},w,{\bf z})$. Another one,
$\tilde H^{(out)}_j$, is parallel to the hyperplane $w=0$ and intersects
connection $\kappa_j$ at the point $w=1$; coordinates in the hyperplane are
$({\bf u},{\bf v},{\bf z})$. Near $\xi_j$ trajectories of the system (\ref{eq_ode})
can be approximated by a local map (called the {\em first return
map})
$\phi_j:\tilde H^{(in)}_j\to \tilde H^{(out)}_j$ relating a point, where a
trajectory enters the neighbourhood, to the point, where it exits.
In the leading order (see (\ref{lmap1})-(\ref{relat})\,), the local map is
\begin{equation}\label{phmap}
\renewcommand{\arraystretch}{1.4}
\begin{array}{l}
({\bf u}^{out},{\bf v}^{out},{\bf z}^{out})=\phi_j({\bf q},w,\{z_s\})=\\
(f^{rad}(A^{rad},{\bf u}_0+P^{rad}D_j{\bf q},w,e_j),
\{(v_{0d}+P^{con}_dD_j{\bf q})w^{-c_{j,d}/e_j}\},\{z_sw^{-t_{j,s}/e_j}\}),
\end{array}
\end{equation}
where $P^{rad}$ is the projection on the radial subspace and
$P^{con}_d$ the projection on the direction of the $d$-th contracting eigenvector.
The expressions for ${\bf z}^{out}$ follows from (\ref{lmap1}), the ones for
${\bf v}^{out}$ from (\ref{lmap1}) and (\ref{relat}).
The mapping $f^{rad}$ provides the radial coordinates for $\phi_j$.
No explicit expressions can be given since the structure of the eigenspaces
of the matrix $A^{rad}$ is unknown. An estimate for the mapping $f^{rad}$
will be proven in section \ref{stabcm}.

Near connection $\kappa_j$ the system (\ref{eq_ode}) can be approximated
by a global map (also called a {\em connecting diffeomorphism})
$\psi_j:\tilde H^{(out)}_j\to \tilde H^{(in)}_{j+1}$,
\begin{equation}\label{cdif}
\left(
\begin{array}{c}
{\bf q}^{j+1}\\
w^{j+1}\\
{\bf z}^{j+1}
\end{array}
\right):=\psi_j
\left(
\begin{array}{c}
{\bf u}^j\\
{\bf v}^j\\
{\bf z}^j
\end{array}
\right)=A^{\rm glob}_jB^{\rm glob}_j
\left(
\begin{array}{c}
{\bf u}^j\\
{\bf v}^j\\
{\bf z}^j
\end{array}
\right),
\end{equation}
where superscripts in the notation of components indicate, whether
the respective vector is decomposed
in the local basis near $\xi_j$ or near $\xi_{j+1}$. The $(n-1)\times(n-1)$
matrix $B^{\rm glob}_j$ presents the map $\psi_j$ in the local coordinates near
$\xi_j$ (i.e., the basis near $\xi_{j+1}$ is the same, as near $\xi_j$,
and the origin is shifted to $\xi_{j+1}$), and matrix $A^{\rm glob}_j$
relates the coordinates in the two local bases).
Lemma \ref{lecol} implies that the lower right
$n_j\times n_j$ block of $B^{\rm glob}_j$ related to the
contracting and transverse eigenvalues of $df(\xi_j)$ is diagonal, while
the lower left $n_j\times (n-n_j-1)$ block vanishes.
Matrix $A^{\rm glob}_j$ has two non-vanishing blocks. The upper left
$(n-n_j-1)\times (n-n_j-1)$ relate the coordinates ${\bf q}^{j+1}$ with coordinates
${\bf u}^j$ and the lower right $n_j\times n_j$ block is a permutation matrix.

Denote the superpositions of the local, $\phi_j$, and global, $\psi_j$, maps by
$\tilde g_j=\psi_j\circ\phi_j:\tilde H^{(in)}_j\to\tilde H^{(in)}_{j+1}$.
The Poincar\'e map
$\tilde H^{(in)}_1\to\tilde H^{(in)}_1$ for the cycle is the superposition
$\tilde\pi_1=\tilde g_m\circ\ldots\circ\tilde g_1$; for $j>1$ the Poincar\'e
maps $\tilde H^{(in)}_j\to\tilde H^{(in)}_j$ are constructed similarly:
$$\tilde\pi_j=\tilde g_{j-1}\circ\ldots\circ\tilde g_1\circ\tilde g_m
\circ\ldots\circ\tilde g_j.$$

Since ${\bf v}_0\gg{\bf q}$, in (\ref{phmap}) in the sum
$v_{0d}+P^{con}_dD_j{\bf q}$ the second term can be ignored.
The coordinates $(w,{\bf z})$ in the maps $\tilde g_j$ are
independent of $\bf q$. Hence, we can define maps $g_j$ which are
restrictions of the maps $\tilde g_j$ into the $(w,{\bf z})$-subspace:
\begin{equation}\label{eq_mapg0a}
g_j(w,{\bf z})=
A_jB_j
\left(
\begin{array}{c}
\{v_{0d}\}w^{-c_{j,d}/e_j}\\
\{z_sw^{-t_{j,s}/e_j}\}
\end{array}
\right).
\end{equation}
Here $A_j$ and $B_j$ are the lower right $n_j\times n_j$
corners of the matrices $A^{\rm glob}_j$ and $B^{\rm glob}_j$.
Recall that $B_j$ is diagonal and $A_j$ is a permutation matrix.

We call the set of maps $\{g_1^m\}=\{g_1,\ldots,g_m\}$, where
$g_j:\R^{n_j}\to\R^{n_j+1}$ have been constructed above,
{\it a collection of maps associated with the heteroclinic cycle}
$\cX=\{\xi_1,\ldots,\xi_m\}$. The collection of maps
$\{g_l^{l-1}\}=\{g_l,\ldots,g_m,g_1,\ldots,g_{l-1}\}$ is associated with
the heteroclinic cycle $\{\xi_l,\ldots,\xi_m,\xi_1,\ldots,\xi_{l-1}\}$
which geometrically coincides with the former cycle.

\subsection{Collection of maps: definitions of stability}
\label{stamap}

Given a collection of maps $\{g_1^m\}=\{g_1,\ldots,g_m\}$,
$g_j:\R^{N_j}\to\R^{N_{j+1}}$, we define
superpositions as follows
\begin{equation}\label{mapg}
\pi_j=g_{j-1}\circ\ldots\circ g_1\circ g_m\circ\ldots\circ g_{j+1}\circ g_j
\end{equation}
(for $j=1$ and $j=2$ this reduces to
$\pi_1=g_m\circ\ldots\circ g_2\circ g_1$ and
$\pi_2=g_1\circ g_m\circ\ldots\circ g_2$, respectively) and
\begin{equation}\label{sup2}
g_{(l,j)}=\left\{
\renewcommand{\arraystretch}{2.0}
\begin{array}{ll}
g_{l}\circ\ldots\circ g_j, & l>j\\
g_{l}\circ\ldots\circ g_m\circ g_1\circ\ldots\circ g_j, & l<j
\end{array}\right..
\end{equation}

We call ${\bf y}^1\in\R^{N_1}$ to be a fixed point of
the collection of maps $\{g_1^m\}$, if $\pi_1{\bf y}^1={\bf y}^1$. Evidently,
${\bf y}^l=g_{(l-1,1)}{\bf y}^1$ is then
a fixed point of the collection of maps
$\{g_l^{m}\}=\{g_l,\ldots,g_m,g_1,\ldots,g_{l-1}\}$.

\begin{definition}\label{def8}
We say that a fixed point ${\bf y}^1\in\R^{N_1}$ of a collection of maps
\begin{equation}\label{colm}
\{g_1^m\}=\{g_1,\ldots,g_m\},\ g_j:\R^{N_j}\to\R^{N_{j+1}},
\end{equation}
is \underline{asymptotically stable}, if for any $\delta>0$ there exists an
$\varepsilon>0$ such that for any $1\le l\le m$
$$d({\bf x},{\bf y}^l)<\varepsilon,\hbox{ where } {\bf y}^l=g_{(l-1,1)}{\bf y}^1,$$
implies
$$d(\pi_j^kg_{(j-1,l)}{\bf x},g_{(j-1,l)}{\bf y}^l)<\delta
\mbox{ for all }\ 1\le j\le m,\ k\ge0$$
and
$$\lim_{k\to\infty}d(\pi_j^kg_{(j-1,l)}{\bf x},g_{(j-1,l)}{\bf y}^l)=0
\mbox{ for all }1\le j\le m.$$
\end{definition}

\begin{definition}\label{def9}
We say that a fixed point ${\bf y}^1\in\R^{N_1}$ of a collection of maps $\{g_1^m\}$
is \underline{fragmentarily}\break\underline{asymptotically stable}, if
for any $\delta>0$
$$\mu(\cB_{\delta}(\{g_1^m\},{\bf y}^1))>0,$$
where
$$\cB_{\delta}(\{g_1^m\},{\bf y}^1):=
\{{\bf x}~:~{\bf x}\in\R^{N_1},\ d(\pi_j^kg_{(j-1,1)}{\bf x},
g_{(j-1,1)}{\bf y}^1)<\delta\mbox{ for all }1\le j\le m,\ k\ge0$$
$$\hbox{and }\lim_{k\to\infty}d(\pi_j^kg_{(j-1,1)}{\bf x},g_{(j-1,1)}{\bf y}^1)=0
\mbox{ for all }1\le j\le m\}.$$
\end{definition}

\begin{definition}\label{def91}
We say that a fixed point ${\bf y}^1\in\R^{N_1}$ of a collection of maps $\{g\}_1^m$ is
\underline{completely unstable}, if there exists $\delta>0$ such that
$$\mu(\cB_{\delta}(\{g_1^m\},{\bf y}^1))=0.$$
\end{definition}

\section{Stability of a cycle and a collection of maps}
\label{stabcm}

In this section we prove that a heteroclinic cycle of type Y is asymptotically
stable, or fragmentarily asymptotically stable, whenever the collection of maps
associated with the cycle has the property. We begin by proving several lemmas.

\begin{lemma}\label{le01}
Let $y(\tau)$ be a solution to the equation
\begin{equation}\label{eql0m}
\dot y(\tau)=Jy(\tau),\ y(\tau)\in\C^n,\ y(0)=y_0,
\end{equation}
where $J$ is a Jordan cell,
\begin{equation}\label{eql0}
J=\left(
\renewcommand{\arraystretch}{1.2}
\begin{array}{cccccc}
\lambda & 1 & 0 & ... & 0 & 0\\
0 & \lambda & 1 & ... & 0 & 0\\
. & . & . & ... & . & .\\
0 & 0 & 0 & ... & \lambda & 1\\
0 & 0 & 0 & ... & 0 & \lambda
\end{array}
\right)
\end{equation}
and $\Re(\lambda)<0$.
Then there exists a constant $K<\infty$, independent of $y_0$, such that
$$|y(\tau)|<Ke^{\Re(\lambda)\tau/2}|y_0|\hbox{ for any }\tau>0.$$
\end{lemma}

\proof
The solution to (\ref{eql0m}) is
$$y(\tau)=e^{J\tau}y_0=e^{\lambda \tau}
\left(
\renewcommand{\arraystretch}{1.2}
\begin{array}{cccccc}
1 & \tau & {\tau^2\over 2!} & ... & {\tau^{n-2}\over(n-2)!} &
{\tau^{n-1}\over(n-1)!}\\
0 & 1 & \tau &  ... & {\tau^{n-3}\over(n-3)!}& {\tau^{n-2}\over(n-2)!}\\
. & . & . & ... & . & .\\
0 & 0 & 0 & ... & 1 & \tau\\
0 & 0 & 0 & ... & 0 & 1
\end{array}
\right)y_0.
$$
Therefore, for $\tau>1$ we have that
$$|y(\tau)|<n\tau^ne^{\Re(\lambda) \tau}|y_0|.$$
The equation $n\tau^ne^{\Re(\lambda) \tau/2}=1$, where $\Re(\lambda)<0$, has a
unique solution $\tau=\tau^*$, $0<\tau^*<\infty$, and
$n\tau^ne^{\Re\lambda \tau/2}<1$ for
any $\tau>\tau^*$. Taking $\tilde \tau=\max(1,\tau^*)$ and
$K=\max_{0\le \tau\le \tilde \tau}n\tau^ne^{\Re(\lambda) \tau/2}$
we prove the lemma.
\qed

\begin{lemma}\label{le02}
Consider functions $x(\tau):\R\to\R$, $y_j(\tau):\R\to\R$, $1\le j\le J$.
Suppose that
$$|x(\tau)|\le\tilde xe^{\alpha \tau},\ |y_j(\tau)|\le\tilde y_je^{\beta_j \tau},
\hbox{ where }\alpha>0,\ \beta_j<0,\ \hbox{ and }\tilde x<\tilde y_j
\hbox{ for any }1\le j\le J.$$
Denote $\tilde y=max_{1\le j\le J}\tilde y_j$,
$\beta=max_{1\le j\le J}\beta_j$ and $C=\tilde y^{\alpha/(\alpha-\beta)}$.
Then there exists $\tau^*>0$ such that
$$|x(\tau)|<C\tilde x^{-\beta/(\alpha-\beta)}\hbox{ for }\tau<\tau^*\hbox{ and }
|y_j(\tau)|<C\tilde x^{-\beta/(\alpha-\beta)}\hbox{ for }\tau>\tau^*
\hbox{ and any }1\le j\le J.$$
\end{lemma}

\proof
Let $\tau^*$ be the solution to the equation
$$\tilde xe^{\alpha \tau}=\tilde ye^{\beta \tau}.$$
Since $\tilde x<\tilde y$, $\alpha>0$ and $\beta<0$, we have that $\tau^*>0$.
The positiveness of $\alpha$ implies
that for any $\tau<\tau^*$
$$|x(\tau)|<\tilde xe^{\alpha \tau^*}=
\tilde y^{\alpha/(\alpha-\beta)}\tilde x^{-\beta/(\alpha-\beta)}=
C\tilde x^{-\beta/(\alpha-\beta)}.$$
The exponents $\beta_j$ are negative, therefore for any $\tau>\tau^*$
$$|y_j(\tau)|< \tilde y_j e^{\beta_j \tau^*}<\tilde ye^{\beta \tau^*}=
\tilde y^{\alpha/(\alpha-\beta)}\tilde x^{-\beta/(\alpha-\beta)}=
C\tilde x^{-\beta/(\alpha-\beta)}.$$
\qed

\begin{lemma}\label{le03}
Consider a trajectory $\Phi(\tau,x_0)$ in the $\delta$-neighbourhood of a
heteroclinic cycle $X$ of type Y. Let $\delta<h$ and in the $h$-neighbourhood of
the equilibria the flow of (\ref{eq_ode}) can be approximated by
(\ref{lmap1}) and (\ref{lmap2}). Suppose that the trajectory intersects
$\tilde H^{in}_j$ at $\tau=\tau^{in}$, $\tilde H^{out}_j$ at $\tau=\tau^{out}$
and for $\tau^{in}\le \tau\le \tau^{out}$ it belongs to the $h$-neighbourhood
of $\xi_j$. The radial eigenvalues of $df(\xi_j)$ have negative
real parts and the contracting eigenvalues are negative.
Denote
$\Phi(\tau^{in},x_0)=x^{in}=({\bf u}^{in},{\bf v}^{in},w^{in},{\bf z}^{in})$
and
$\Phi(\tau^{out},x_0)=x^{out}=({\bf u}^{out},{\bf v}^{out},w^{out},{\bf z}^{out})$.
Then
\begin{itemize}
\item[I] $|w^{in}|<\varepsilon$ implies that
$\max(|{\bf u}^{out}|,|{\bf v}^{out}|)<C\varepsilon^{\gamma}$, where
$\gamma>0$ and $C<\infty$ are independent of $\varepsilon$.
\item[II] Suppose that the transverse eigenvalues of $df(\xi_j)$ are negative
and consider $\tilde x=\Phi(\tilde \tau,x_0)$, where
$\tau^{in}\le \tilde \tau\le \tau^{out}$. If
$d(\tilde x,\cX)<\varepsilon$ then
$d(\Phi(\tau,x_0),\cX)<C\varepsilon^{\gamma}$ for any
$\tilde \tau\le \tau\le \tau^{out}$, where $\gamma>0$ and $C<\infty$ are
independent of $\varepsilon$ and $\tilde \tau$.
\item[III] If
$d(x^{in},\cX)<\varepsilon$ and $d(x^{out},\cX)<\varepsilon$ then
$d(\Phi(\tau,x_0),\cX)<C\varepsilon^{\gamma}$ for any
$\tau^{in}\le \tau\le \tau^{out}$,
where $\gamma>0$ and $C<\infty$ are independent of $\varepsilon$.
\end{itemize}
\end{lemma}

\proof
I. According to (\ref{lmap1}) and (\ref{lmap2}), the coordinates ${\bf u}$ and ${\bf v}$ satisfy
$$\dot {\bf u}=A^{rad}{\bf u}\hbox{ and }\dot v_d=c_{j,d} v_d.$$
Denote by $r_j$ the maximal real part of the eigenvalues of $A^{rad}$.
The radial subspace can be splitted into a direct sum of generalised
eigenspaces of the matrix. Applying lemma \ref{le01} in each of these
subspaces we obtain that
\begin{equation}\label{est1le3}
|{\bf u}^{out}|<Ke^{r_j(\tau^{out}-\tau^{in})/2}|{\bf u}^{in}|.
\end{equation}
The contracting coordinates satisfy
\begin{equation}\label{est2le3}
v_d^{out}=e^{c_{j,d}(\tau^{out}-\tau^{in})}v_d^{in}<
e^{c_j(\tau^{out}-\tau^{in})}v_d^{in},
\end{equation}
where $c_j$ is the maximal contracting eigenvalue.
From $h=w^{in}e^{e_j(\tau^{out}-\tau^{in})}$ we obtain that
$e^{(\tau^{out}-\tau^{in})}=h^{1/e_j}(w^{in})^{-1/e_j}$.
Therefore, due to $|{\bf u}^{in}|<h$ and $|{\bf v}^{in}|<h$
$$|{\bf u}^{out}|<Kh^{1+r_j/2e_j}(w^{in})^{-r_j/2e_j}\hbox{ and }
v_d^{out}<h^{1+1/e_j}(w^{in})^{-c_j/r_j}.$$
Taking $\gamma=\min(-r_j/2e_j,-c_j/r_j)$ and
$C=\max(Kh^{1+r_j/2e_j},h^{1+1/e_j})$ we prove part I of the lemma.

\medskip\noindent
II. The condition $d(\tilde x,\cX)<\varepsilon$, where
$\tilde x\in B_h(\xi_j)$, implies that either
$d(\tilde x,\kappa_{j-1})<\varepsilon$ or $d(\tilde x,\kappa_j)<\varepsilon$.
In the latter case, due to (\ref{lmap1}), (\ref{lmap2}) and lemma \ref{le01} we have that
for $\tilde \tau<\tau<\tau^{out}$
$$d(\Phi(\tau,x_0),\kappa_j)<Ke^{r_j\hat\tau/2}|\tilde {\bf u}|
+e^{c_j\hat\tau}|\tilde {\bf v}|+e^{t_j\hat\tau}|\tilde {\bf z}|,$$
where $t_j$ is the maximal transverse eigenvalue,  $\hat\tau=\tau-\tilde\tau$, and
$(\tilde {\bf u},\tilde {\bf v},\tilde w,\tilde {\bf z})$ are the coordinates
of $\tilde x$. Since all exponents are negative, the statement of the
lemma holds true with $\gamma=1$ and $C=\max(1,K)$.

In the former case using  (\ref{lmap1}), (\ref{lmap2}) and lemma \ref{le01} we write
\begin{equation}\label{est3le3}
d(\Phi(\tau,x_0),\kappa_{j-1})<Ke^{r_j\hat\tau/2}|\tilde {\bf u}_0-\tilde{\bf u}|
+e^{c_j\hat\tau}|\tilde {\bf v}_0-\tilde{\bf v}|+e^{e_j\hat\tau}|\tilde w|+
e^{t_j\hat\tau}|\tilde {\bf z}|
\end{equation}
and
\begin{equation}\label{est4le3}
d(\Phi(\tau,x_0),\kappa_j)<Ke^{r_j\hat\tau/2}|\tilde {\bf u}|
+e^{c_j\hat\tau}|\tilde {\bf v}|+e^{t_j\hat\tau}|\tilde {\bf z}|.
\end{equation}
In the first inequality $(\tilde {\bf u}_0,\tilde {\bf v}_0)$ is the
nearest to $\tilde x$ point of $\kappa_{j-1}$.
In the expressions
(\ref{est3le3}) and (\ref{est4le3}) the last term is less then $\varepsilon$
because $|\tilde {\bf z}|<\varepsilon$ and the transverse eigenvalues are
negative. The first two terms in
(\ref{est3le3}) are less then $(K+1)\varepsilon$ because the radial
eigenvalues have negative real parts and the contracting ones are negative.

By lemma \ref{le02} there exists $\tau^*>\tilde \tau$ such that for $\tau>\tau^*$
$$Ke^{r_j\hat\tau/2}|\tilde {\bf u}|+e^{c_j\hat\tau}|\tilde {\bf v}|<C^*\varepsilon^{\gamma^*},$$
while for $\tilde\tau<\tau<\tau^*$
$$e^{e_j\hat\tau}|\tilde w|<C^*\varepsilon^{\gamma^*},$$
where $\gamma^*>0$ and $C^*$ is a constant, independent of $\varepsilon$.
Taking $C=C^*+K+1$ and $\gamma=\max(1,\gamma^*)$ we prove part II of
the lemma.

\medskip\noindent
III. Set $\tilde\tau=\tau^{in}$ and replace the last term in (\ref{est3le3})
and (\ref{est4le3}) by the sum
$$\sum_{1\le s\le n_t,\ t_{j,s}<0}e^{t_{j,s}(\tau-\tau^{in})}|z_s^{in}|+
\sum_{1\le s\le n_t,\ t_{j,s}\ge 0}e^{t_{j,s}(\tau-\tau^{out})}|z_s^{out}|,$$
where the coordinates corresponding to negative transverse eigenvalues
go to the first term and corresponding to the positive ones to the second.
By the conditions of the lemma at $\tau=\tau^{in}$ and $\tau=\tau^{out}$ both
sums are less then $\varepsilon$, implying that they are less then $\varepsilon$
at times between. The rest of the proof is the same as for part II.
\qed

\begin{theorem}\label{th1}
Let $\{g_1^m\}$, $g_j:\R^{n_j}\to\R^{n_{j+1}}$, be the collection of maps
associated with a type Y heteroclinic cycle $\cX$. For any $\xi_j\in \cX$ the
radial eigenvalues of $df(\xi_j)$ have negative real parts and the
contracting and transverse eigenvalues are negative. Then the cycle is
asymptotically stable if and only if the fixed point $(w,{\bf z})=\bf 0$
of the collection of maps is asymptotically stable.
\end{theorem}

\proof
Suppose the cycle is asymptotically stable. For a given $\delta>0$ denote
by $\varepsilon^*$ the value of $\varepsilon$ that satisfies definition
\ref{def1}. Consider
$x^{in}=({\bf u}^{in},{\bf v}^{in},w^{in},{\bf z}^{in})$ that belongs
to a crossection $\tilde H^{(in)}_j$ and suppose that
$|(w^{in},{\bf z}^{in})|<\varepsilon=\varepsilon^*/2$,
$|{\bf u}^{in}|<\varepsilon^*/4$ and $|{\bf v}^{in}|<\varepsilon^*/4$,
i.e., $|x^{in}|<\varepsilon^*$. A trajectory through $x^{in}$ remains
in the $\delta$-neighbourhood of the cycle for any positive $\tau$. In particular,
for any intersection with any $\tilde H^{(in)}_i$ at the point of intersection
the coordinates $w^{in}$ and ${\bf z}^{in}$ satisfy
$|(w^{in},{\bf z}^{in})|<\delta$ and in the limit $\tau\to\infty$ the norm
$|(w^{in},{\bf z}^{in})|$ vanishes. For small $\delta$ the collection of
maps accurately approximate the flow near the cycle. Therefore, the
fixed point of the collection of maps is asymptotically stable.

\medskip
Suppose that the fixed point of the collection of maps is asymptotically stable.
Let $\delta^*$ and $\varepsilon^*$ be the $\delta$ and $\varepsilon$ of
definition \ref{def8}. Consider
$x^{in}=({\bf u}^{in},{\bf v}^{in},w^{in},{\bf z}^{in})\in\tilde H^{(in)}_j$,
$d(x^{in},\cX)<\varepsilon^*$. Denote by $x^{in}_{ik}$ and $x^{out}_{ik}$
the $k$-th crossings of $\Phi(\tau,x^{in})$ with $\tilde H^{(in)}_i$ and
$\tilde H^{(out)}_i$, respectively, and by $\tau_{ik}^{(in)}$ and
$\tau_{ik}^{(out)}$ the times of the crossings. Asymptotic stability of
the collection of maps implies that
\begin{equation}\label{eqt11}
|(w^{in}_{ik},{\bf z}^{in}_{ik})|<\delta^*
\end{equation}
and
\begin{equation}\label{eqt12}
\lim_{\tau\to\infty}(w^{in}_{ik},{\bf z}^{in}_{ik})=0.
\end{equation}

By lemma \ref{le03}I for any $i$ and $k$
\begin{equation}\label{eqt13}
|({\bf u}^{out}_{ik},{\bf v}^{out}_{ik})|<C_1(\delta^*)^{\gamma_1},
\end{equation}
while from (\ref{lmap1})
\begin{equation}\label{eqt14}
|{\bf z}^{out}_{ik}|<C_2(\delta^*)^{\gamma_2},
\end{equation}
where $C_1,C_2<\infty$ and $\gamma_1,\gamma_2>0$. Hence,
\begin{equation}\label{eqt15}
d(x^{out}_{ik},\cX)<C_3(\delta^*)^{\gamma_3},
\end{equation}
where $C_3=C_1+C_2$ and $\gamma_3=\min(\gamma_1,\gamma_2)$.

Smooth dependence of $\Phi(\tau,x)$ on time and initial condition implies
existence of a constant $\tilde C$ such that
\begin{equation}\label{eqt16}
d(\Phi(\tau,x^{in}),\cX)<\tilde C d(\Phi(\tau^{out}_{ik},x^{in}),\cX)
<\tilde CC_3(\delta^*)^{\gamma_3}
\hbox{ for any }\tau^{out}_{ik}\le \tau\le \tau^{in}_{i+1,k}.
\end{equation}
Since (\ref{eqt15}) and (\ref{eqt16}) hold true for any $\tau^{in}_{ik}$
and $\tau^{out}_{ik}$, by lemma \ref{le03}II
\begin{equation}\label{eqt17}
d(\Phi(\tau,x^{in}),\cX)<C_4(\delta^*)^{\gamma_4}
\hbox{ for any }\tau^{in}_{ik}\le \tau\le \tau^{out}_{ik},
\end{equation}
where $\gamma_4=\gamma_3\hat\gamma$,
$C_4=\hat C(\max(\tilde CC_3,C_3))^{\hat\gamma}$ and $\hat C$ and $\hat\gamma$
are the constants of lemma \ref{le03}II.

Suppose that $\delta^*>0$ satisfies $\tilde CC_3(\delta^*)^{\gamma_3}<\delta$
and $C_4(\delta^*)^{\gamma_4}<\delta$.
Then by (\ref{eqt16}) and (\ref{eqt17}) a trajectory
through $x^{in}\in\tilde H^{(in)}_j$, $d(x^{in},\cX)<\varepsilon^*$,
for any $\tau>0$ remains in the $\delta$-neighbourhood of the cycle
is attracted to the cycle as $\tau\to\infty$.

Next we consider $x^{in}\not\in\tilde H^{(in)}_j$, $d(x^{in},\cX)<\varepsilon^*$.
Suppose that $x^{in}\not\in B_h(\xi_j)$ for any $j$ and $\kappa_i$ is
the nearest to $x^{in}$ connection of $\cX$. The flow $\Phi(\tau,x)$
near $\kappa_i$ between $\tilde H^{(out)}_i$ and $\tilde H^{(in)}_{i+1}$ can be
approximated by a linear map. Therefore, there exists a constant $K$ such that
\begin{equation}\label{eqt18}
d(\Phi(\tau^{in}_{i+1},x^{in}),\kappa_i)<Kd(x^{in},\kappa_i)
\end{equation}
where $\tau^{in}_{i+1}$ is the time when the trajectory crosses
$\tilde H^{(in)}_{i+1}$.

If $x^{in}\in B_h(\xi_j)$ for some $j$, then by lemma \ref{le03}II
\begin{equation}\label{eqt19}
d(\Phi(\tau^{out}_j,x^{in}),\cX)<C(d(x^{in},\cX))^{\gamma},
\end{equation}
where $\tau^{out}_j$ is the time when the trajectory crosses $\tilde H^{(out)}_j$.
Let $\varepsilon$ satisfies
$K\varepsilon<\varepsilon^*$ and $KC\varepsilon^{\gamma}<\varepsilon^*$
and $d(x^{in},\cX)<\varepsilon$.
Denote by $\tau^{in}$ the time of the first crossing of $\Phi(\tau,x^{in})$
with any of $\tilde H^{(in)}_j$.
From (\ref{eqt18}) and (\ref{eqt19}), at the crossing
$d(\Phi(\tau^{in},x^{in}),\cX)<\varepsilon^*$.
Hence, as proven above the trajectory stays in the neighbourhood of the
cycle and is attracted by the cycle at large $\tau$.
\qed

\begin{lemma}\label{le11}
Let $\cX$ be a heteroclinic cycle, comprised of equilibria and heteroclinic
trajectories, in a smooth dynamical system in $\R^n$.
If $\cX$ is fragmentarily asymptotically
stable then for any $\delta>0$ and any $n-1$-dimensional hyperplane
$H$ intersecting a connection $\kappa_j$
\begin{equation}\label{clem1}
\mu^{n-1}(H\cap \cB_{\delta}(\cX))>0,
\end{equation}
where $\mu^{n-1}$ is the Lebesgue measure in $\R^{n-1}$.
\end{lemma}

\proof
Denote $Q=H\cap \cB_{\delta}(\cX)$ and suppose that $\mu^{n-1}Q=0$.
The set $\cB_{\delta}(X)$ can be represented as a union of segments
of trajectories going through points $x_0\in Q$ at $\tau=0$ in the direction
of negative $\tau$,
$$\cB_{\delta}(X)=\cup_{x_0\in Q}\Phi([\tau^{-},0],x_0),$$
where $\Phi([\tau^{-},0],x_0)$ is the segment of trajectory
$\Phi(\tau,x_0)$ between $\tau^{-}<0$ and $0$.
At time $\tau=\tau^{-}$ the trajectory either leaves the $\delta$-heighbourhood
of $\cX$, or it crosses $H$, or $\tau^{-}=-\infty$ and the trajectory
belongs to the unstable manifold of one of $\xi_j$. Respectively, we
split $\cB_{\delta}(\cX)=\cB^{fin}+\cB^{inf}$, where the set $\cB^{fin}$ is
comprised of trajectories with finite $\tau^{-}$ and the set $\cB^{inf}$ of
the ones with infinite $\tau^{-}$. The measure of $\cB^{inf}$ is zero,
because the unstable manifolds of $\xi_j$ have zero measure. For a $C<0$ denote
$$\cB_C=\cup_{x_0\in Q,\ \tau^{-}>C}\Phi([\tau^{-},0],x_0).$$
For any value of $C$ the measure of $\cB_C$ is zero. Hence,
$$\mu\cB^{fin}=\lim_{C\to-\infty}\cB_C=0.$$
Therefore, $\mu\cB_{\delta}(\cX)=\mu\cB^{fin}+\mu\cB^{inf}=0$, which contradicts to
the conditions of the lemma.
\qed

\begin{theorem}\label{th2}
Let $\{g_1^m\}$, $g_j:\R^{n_j}\to\R^{n_{j+1}}$, see (\ref{eq_mapg0a}),
be the collection of maps associated with a type Y heteroclinic cycle $\cX$.
Suppose that the diagonal matrices $B_j$ are non-degenerate.
\footnote
{The entries of $B_j$ depend on $h$, the size of the neighbourhoods of $\xi_j$
employed in the construction of local and global maps.
Since the condition that $B_j$ are non-degenerate generically holds true,
we can assume $h$ to be chosen such that the property is satisfied.}
The cycle is fragmentarily
asymptotically stable if and only if the fixed point $(w,{\bf z})=\bf 0$
of the collection of maps is fragmentarily asymptotically stable.
\end{theorem}

\proof
Suppose that the cycle is fragmentarily asymptotically stable. For a
given $\delta>0$ denote
$$Q=P_{w,{\bf z}}(\tilde H_1^{(in)}\cap\cB_{\delta}(\cX)),$$
where $P_{w,{\bf z}}$ is the projection on the subspace spanned by
$w$ and ${\bf z}$. Due to lemma \ref{le11} the measure $\mu^{n-1}$ of the set
$\tilde H_1\cap\cB_{\delta}(X)$ is positive. Therefore, $\tilde\mu Q>0$,
where $\tilde\mu$ denotes a measure in the subspace $(w,{\bf z})$.
The collection of maps accurately approximates trajectories in the vicinity of
the cycle, at each intersection of $\tilde H_j^{(in)}$ the trajectories stay in
the neighbourhood of the cycle and are attracted as $\tau\to\infty$. Therefore,
$Q\subset\cB_{\delta}(\{g_1^m\},{\bf y}^1)$.

\medskip
Suppose the fixed point of the collection of maps is fragmentarily
asymptotically stable. Denote $Q=\cB_{\delta^*}(\{g_1^m\},{\bf y}^1)$ and
$$W=\{~({\bf u},{\bf v},w,{\bf z})\ :\ |{\bf u}-{\bf u}_0|<\delta^*,\
|{\bf v}-{\bf v}_0|<\delta^*,\ (w,{\bf z})\in Q\},$$
where the coordinates are taken in the basis comprised of eigenvectors of
$df(\xi_1)$ and $({\bf u}_0,{\bf v}_0)$ are the coordinates of the intersection
of $\kappa_m$ with $\tilde H_1^{(in)}$. Consider a trajectory $\Phi(\tau,x_0)$
through $x_0\in W$. Asymptotic stability
of the collection of maps implies that at the $k$-th intersection with
$\tilde H_j^{(in)}$ we have that
\begin{equation}\label{th21}
|(w^{in}_{jk},{\bf z}^{in}_{jk})|<\delta^*\hbox{ for all }
1\le j\le m\hbox{ and }1\le k<\infty.
\end{equation}
Lemma \ref{le03}I implies that
\begin{equation}\label{th22}
|{\bf u}^{out}_{jk},{\bf v}^{out}_{jk}|<C_1(\delta^*)^{\gamma_1},
\end{equation}
where $C_1$ and $\gamma_1$ are the maximal over $j$ values of $C$ and $\gamma$.

Let
$$C^*=\max_{1\le i\le n_j,\ 1\le j\le m}{1\over |b_{ii}^j|},$$
where $b_{ii}^j$ are the entries of $B_j$. Since the matrices are
non-degenerate, $C^*$ is finite. The global maps are predominantly
linear, therefore
\begin{equation}\label{th23}
|{\bf z}^{out}_{jk}|<C^*|{\bf z}^{in}_{j+1,k}|<C^*\delta^*.
\end{equation}
From (\ref{th22}) and (\ref{th23}) we have that
\begin{equation}\label{th24}
d(x^{out}_{jk},\cX)<C_2(\delta^*)^{\gamma_2},
\end{equation}
where $x^{out}_{jk}$ is the $k$-th intersection of $\Phi(\tau,x_0)$ with
$\tilde H^{(out)}_j$, $C_2=C_1\max(1,C^*)$ and $\gamma_2=\min(1,\gamma_1)$.

Since $\Phi(\tau,x)$ is a smooth function of time and initial condition,
there exists a constant $\hat C$ such that
\begin{equation}\label{th25}
d(\Phi(\tau,x_0),\cX)<\hat Cd(x^{out}_{jk},\cX)<\hat CC_2(\delta^*)^{\gamma_2}
\hbox{ for any }\tau^{out}_{jk}\le \tau\le \tau^{in}_{j+1,k}.
\end{equation}
In particular,
\begin{equation}\label{th26}
d(x^{in}_{j+1,k},\cX)<\hat CC_2(\delta^*)^{\gamma_2}.
\end{equation}

The inequalities (\ref{th24}) and (\ref{th26}) hold true for any $1\le j\le m$
and $k>0$. Applying lemma \ref{le03}III to $x^{in}_{jk}$ and $x^{out}_{jk}$,
we obtain that
\begin{equation}\label{th27}
d(\Phi(\tau,x_0),\cX)<C_3(\delta^*)^{\gamma_3}
\hbox{ for any }\tau^{in}_{jk}\le \tau\le \tau^{out}_{jk},
\end{equation}
where $\gamma_3=\gamma_2\tilde\gamma$,
$C_3=\tilde C(\max(\hat CC_2,C_2))^{\tilde\gamma}$ and $\tilde C$ and
$\tilde\gamma$ are the constants of lemma \ref{le03}III.

From (\ref{th25}) and (\ref{th27}), if $\delta^*$ satisfies
$\hat CC_2(\delta^*)^{\gamma_2}<\delta$ and $C_3(\delta^*)^{\gamma_3}<\delta$
then the trajectory $\Phi(\tau,x_0)$ stays in the
$\delta$-neighbourhood of $\cX$. Since for $k\to\infty$ the values of
$(w^{in}_{jk},{\bf z}^{in}_{jk})$ tend to zero, the trajectory is approaching
the cycle for $\tau\to\infty$.

We have proved that $W\subset\cB_{\delta}(\cX))$. Since
$\mu(W)=(\delta^*)^{n_1^r+n_1^c}\tilde\mu(Q)>0$ the cycle $\cX$ is f.a.s.
\qed

\section{Stability of fixed points of a collection of maps}\label{smaps}

The main results of this section are theorems \ref{th_3} and \ref{th_4}
that prove necessary and sufficient conditions for asymptotic stability
and fragmentary asymptotic stability of a fixed point of a collection of
maps. The presentation follows \cite{op12} where the stability was studied for
a collection of maps associated with a type Z heteroclinic cycle.
The stability depends on the eigenvalues and eigenvectors of transition
matrices that are products of basic transition matrices. The basic transition
matrices for cycles of types Z and Y differ, however the main results for
type Z cycles hold true for type Y cycles. For lemma \ref{lem_2} and
theorems \ref{th_22}-\ref{th_4} no proofs are given since they can be obtained
from the ones of \cite{op12} by minor modifications.

\subsection{Transition matrix}

In the coordinates \mbel,
\begin{equation}\label{newc}
\mbe=(\ln|w|,\ln|z_1|,...,\ln|z_{n_j^t}|),
\end{equation}
the maps $g_j$, that we denote by ${\cal M}_j$, are
\begin{equation}\label{fmap}
{\cal M}_j\mbe=M_j\mbe+F_j,
\end{equation}
where
\begin{equation}\label{esm}
M_j:=A_jB_j=A_j\left(
\begin{array}{ccccc}
b_{j,1}&0&0&\ldots&0\\
.&.&.&\ldots&.\\
b_{j,n_j^c}&0&0&\ldots&0\\
b_{j,n_j^c+1}&1&0&\ldots&0\\
b_{j,n_j^c+2}&0&1&\ldots&0\\
.&.&.&\ldots&.\\
b_{j,n_{j+1}}&0&0&\ldots&1
\end{array}
\right)
\end{equation}
are the {\em basic transition matrices} of the maps. Here $A_j$ is
an $n_{j+1}\times n_{j+1}$ permutation matrix and $B_j$ is an
$n_{j+1}\times n_j$ matrix. (Recall that $n_{j+1}=n_{j+1}^t+1=n_j^c+n_j^t$.)
The entries $b_{j,l}$ of the matrix $B_j$ depend on the eigenvalues of the
linearisation $df(\xi_j)$ of (\ref{eq_ode}) near $\xi_j$ as follows
\begin{equation}\label{coeB}
b_{j,d}=c_{j,d}/e_j,\ 1\le d\le n_j^c,\mbox{ and }
b_{j,l+n_j^c}=-t_{j,l}/e_j,\ 1\le l\le n_j^t,\ 1\le j\le m.
\end{equation}
We call $\{{\cal M}_1^m\}$, as $\{g_1^m\}$, a collection of
maps associated with the heteroclinic cycle. A fixed point $(w,{\bf z})=\bf 0$
of the collection $\{g_1^m\}$ becomes a fixed point $\mbe=-\mbi$
of the collection $\{{\cal M}_1^m\}$. In the study of stability of the point
$(w,{\bf z})=\bf 0$ we consider asymptotically small $w$ and $\bf z$, i.e.,
asymptotically large negative \mbel, and hence finite $F_j$ can be ignored.

Transition matrices of the superposition of maps $\pi_j$ and $g_{j,l}$
are the products $M^{(j)}=M_{j-1}\ldots M_1M_m\ldots M_{j+1}M_j$ and
$M_{j,l}=M_j\ldots M_1M_m\ldots M_{l+1}M_l$ (or $M_j\ldots M_l$ if $j>l$),
respectively. The matrices $M^{(j)}$ are square of the size
$n_j\times n_j$, while $M_{j,l}$ are of the size
$n_{j+1}\times n_l$.
The matrices $M^{(j)}$ are of the same rank. As proven in appendix A, in
general, the rank is $\min_{1\le j\le m}n_j$.

\subsection{Two types of eigenvalues of a transition matrix}
\label{twot}

Consider a matrix $M:=M^{(1)}=M_m\ldots M_1:\R^N\to\R^N$; it is
a product of the basic transition matrices of the form (\ref{esm}).
We separate
the coordinate vectors ${\bf e}_l$, $1\le l\le N$, into two groups.
A vector ${\bf e}_l$ belongs to the second group if
$A_jA_{j-1}...A_2A_1{\bf e}_l$ is a transverse eigenvector of $df(\xi_j)$ for
any $j$, and to the first one otherwise.
Denote by $V^{\rm sig}$ and $V^{\rm ins}$ the subspaces spanned
by vectors from the first and second group, respectively (the superscripts
``ins'' and ``sig'' stand for significant and insignificant).

\begin{theorem}\label{th_22} (Theorem 3 in \cite{op12}.)
Let $V^{\rm sig}$ and $V^{\rm ins}$ be the subspaces defined above.

\begin{itemize}
\item[(a)] The subspace $V^{\rm ins}$ is $M$-invariant and the absolute value
of all eigenvalues associated with the eigenvectors from this subspace is one.
\item[(b)] Generically all components of eigenvectors that do not
belong to $V^{\rm ins}$ are non-zero.
\end{itemize}
\end{theorem}

\subsection{Properties of maps}

The Poincar\'e maps ${\cal M}^{(j)}$ are superpositions of maps (\ref{fmap}):
\begin{equation}\label{fmap1}
{\cal M}^{(j)}=
{\cal M}_{j-1}\ldots{\cal M}_1{\cal M}_m\ldots {\cal M}_{j+1}{\cal M}_j.
\end{equation}
In the coordinates $\mbe$ (\ref{newc}) they reduce to
\begin{equation}\label{fmapp}
{\cal M}^{(j)}\mbe=M^{(j)}\mbe+{\bf C}^{(j)}.
\end{equation}

For a linear map $\cal M$, where
\begin{equation}\label{fmap0}
{\cal M}\mbe=M\mbe+{\bf C},
\end{equation}
we define
$$
U^{-\infty}({\cal M})=
\{{\bf y}:\ {\bf y}\in\R^N_-,\ \lim_{k\to\infty}{\cal M}^k{\bf y}=-\mbi\}.
$$

\begin{lemma}\label{lem_2}(Lemma 5 in \cite{op12}.)
Let $\lambda_{\max}$ be the largest in absolute value significant eigenvalue
of the matrix $M$ in (\ref{fmap0}) and ${\bf w}^{\max}$ be the associated
eigenvector. Suppose $\lambda_{\max}\ne 1$ (as noted in subsection \ref{twot},
generically this is true). The measure $\mu(U^{-\infty}({\cal M}))$ is
positive, if and only if the three following conditions are satisfied:
\begin{itemize}
\item[(i)] $\lambda_{\max}$ is real;
\item[(ii)] $\lambda_{\max}>1$;
\item[(iii)] $w_l^{\max}w_q^{\max}>0$ for all $l$ and $q$, $1\le l,q\le N$.
\end{itemize}
\end{lemma}

Below $\lambda_{\max}\ne1$ denotes the largest in absolute value significant eigenvalue of
the transition matrices $M^{(j)}$.

\begin{theorem}\label{th_3}(Theorem 4 in \cite{op12}.)
Let $M_j$ be basic transition matrices of a collection of maps $\{g_1^m\}$
associated with a heteroclinic cycle of type Y. Suppose that for all $j$,
$1\le j\le m$, all transverse eigenvalues of $df(\xi_j)$ are negative. Then
\begin{itemize}
\item[(a)] If the inequality $\lambda_{\max}>1$ holds true
for the transition matrix $M:=M^{(1)}=M_m\ldots M_1$, then $\bf 0$ is
an asymptotically stable fixed point of the collection of maps $\{g_1^m\}$.
\item[(b)] If $\lambda_{\max}<1$, then $\bf 0$ is completely unstable.
\end{itemize}
\end{theorem}

\begin{theorem}\label{th_4}(Theorem 5 in \cite{op12}.)
Let $M_j$ be basic transition matrices of a collection of maps $\{g_1^m\}$
associated with a heteroclinic cycle of type Y.
(For type Y heteroclinic cycles the matrices are of the form (\ref{esm}).)
Denote by $j=j_1,\ldots j_L$ the indices, for which $M_j$ involves negative
entries; all entries are non-negative for all remaining $j$. Assume $L>0$
(the case $L=0$ is treated by theorem \ref{th_3}).
\begin{itemize}
\item[(a)] If for at least one $j=j_l+1$ the matrix $M^{(j)}$ does not satisfy
conditions (i)-(iii) of lemma \ref{lem_2}
then $\bf 0$ is a completely unstable fixed point of the collection $\{g_1^m\}$.
\item[(b)] If the matrices $M^{(j)}$ satisfies conditions (i)-(iii) of lemma
\ref{lem_2} for all $j$ such that $j=j_l+1$,
then $\bf 0$ is a fragmentarily
asymptotically stable fixed point of the collection $\{g_1^m\}$.
\end{itemize}
\end{theorem}

\section{Examples of type Y cycles}\label{exa}

In this section we present two examples of type Y heteroclinic cycles
admitted by the generalised Lotka-Volterra (GLV) system in $\R^5_+$:
\begin{equation}\label{lvsys}
\dot x_i=x_i(r_i+\sum_{j=1}^5 a_{ij}x_j),\quad 1\le i\le 5,\quad x_i\ge0.
\end{equation}
The cycles are comprised of equilibria and trajectories that belong to
two- or three-dimensional
invariant subspaces. Dynamics of the Lotka-Volterra system in dimension two is
well known. Dynamics in three-dimensional subspaces, including the existence
of heteroclinic connections, is known only in some particular cases, see
e.g., \cite{op23a} and references therein.
In appendix B we prove sufficient conditions on the coefficients of (\ref{lvsys})
for the existence of heteroclinic connections in three-dimensional
subspaces (see figure \ref{fig1}) emerging in the examples.

\begin{figure}
{\large
\hspace*{-5mm}\includegraphics[width=99mm]{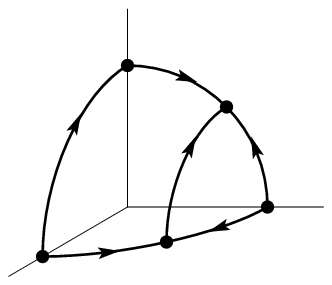}~\hspace*{-25mm}\includegraphics[width=99mm]{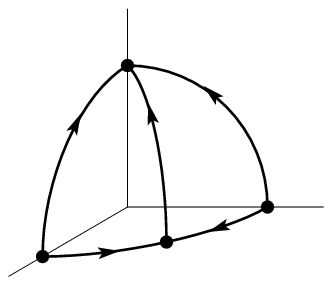}

\vspace*{-37mm}\hspace*{20mm}$\xi_1$\hspace*{71mm}$\xi_1$

\vspace*{-16mm}\hspace*{66mm}$\xi_2$\hspace*{71mm}$\xi_2$

\vspace*{-29mm}\hspace*{44mm}$\xi_3$\hspace*{71mm}$\xi_3$

\vspace*{31mm}\hspace*{50mm}$\xi_1^*$\hspace*{71mm}$\xi_2^*$

\vspace*{-35mm}\hspace*{61mm}$\xi_2^*$

\vspace*{9mm}\hspace*{76mm}$x_2$\hspace*{71mm}$x_2$

\vspace*{10mm}\hspace*{12mm}$x_1$\hspace*{71mm}$x_1$

\vspace*{-52mm}\hspace*{31mm}$x_3$\hspace*{71mm}$x_3$

\vspace*{47mm}\hspace*{25mm}(a)\hspace*{71mm}(b)

\vspace{4mm}
}

\caption{Heteroclinic connections in a three-dimensional GLV system
under the conditions of theorem \ref{tlv30} (a) and theorem \ref{tlv3} (b).
\label{fig1}}
\end{figure}

In both example we set
\begin{equation}\label{lvcoe}
r_i=1\hbox{ and }a_{ii}=-1,
\end{equation}
which implies existence of equilibria $\xi_i$ on each coordinate axis with
$x_i=1$. The equilibria are stable along the respective axes.

\subsection{Example 1}\label{exa1}

We consider system (\ref{lvsys}) equivariant with respect to the symmetry
permuting the coordinate axes,
$\gamma:(x_1,x_2,x_3,x_4,x_5)\to(x_2,x_3,x_4,x_5,x_1)$.
Denote by $\lambda_{ij}$ the eigenvalue of $df(\xi_i)$ in the direction
${\bf e}_j$. From (\ref{lvsys}), (\ref{lvcoe}) we have that the eigenvalues of
$df(\xi_1)$ are
$$\lambda_{12}=1+a_{21},\ \lambda_{13}=1+a_{31},\
\lambda_{14}=1+a_{41},\ \lambda_{15}=1+a_{51}.$$
For other $\xi_i$ the eigenvalues can be obtained applying the symmetry
$\gamma$. The equivariance of the system implies that $a_{ij}=a_{i+1,j+1}$ and
$\lambda_{ij}=\lambda_{i+1,j+1}$. In particular, $a_{21}=a_{15}$,
$a_{31}=a_{14}$, $a_{41}=a_{13}$ and $a_{51}=a_{12}$.

We assume that $1+a_{21}>0$, $1+a_{12}>0$ and  $1-a_{12}a_{21}>0$. Then
(see theorem \ref{tlv2})
there exists an equilibrium $\xi^*_1=(x_1^*,x_2^*,0,0,0)$ with
\begin{equation}\label{exa10}
x_1^*={1+a_{12}\over 1-a_{12}a_{21}},\quad
x_2^*={1+a_{21}\over 1-a_{12}a_{21}}.
\end{equation}
The equilibria $\xi^*_j$ are $\xi^*_j=\gamma^{j-1}\xi^*_1$, where $j=2,3,4,5$.
Denote by $\mu_{ij}$, $j=3,4,5$, the eigenvalues of
$df(\xi^*_i)$ in the direction ${\bf e}_j$. We have
\begin{equation}\label{exa11}
\mu_{13}=1+x^*_1a_{31}+x^*_2a_{32},\
\mu_{14}=1+x^*_1a_{41}+x^*_2a_{42},\ \mu_{15}=1+x^*_1a_{51}+x^*_2a_{52}.
\end{equation}

The cycle of this example is homoclinic. It is comprised
of $\xi^*_1$, heteroclinic connection $\kappa_1:\xi^*_1\to\xi^*_2$
(see figure \ref{fig1}a) and
their images under the action of the symmetry $\gamma$. Hence, for
$\xi^*_1$ the expanding eigenvalue is $\mu_{13}$, the contracting is $\mu_{15}$
and the transverse one is $\mu_{14}$. By construction, the expanding eigenvalue is
positive and the contracting is negative. We assume that the transverse one is
negative, which is a necessary condition for a heteroclinic cycle to be
asymptotically stable. As proven in the
appendix, sufficient conditions for the existence of the connection
$\kappa_1$ in the subspace $(x_1,x_2,x_3,0,0)$ in system (\ref{lvsys}),(\ref{lvcoe})
are
$$
\begin{array}{l}
-1<a_{12}<0,\ 1+a_{13}<0,\ 1+a_{23}>0,\ 1+a_{21}>0,\\
a_{12}a_{21}<1,\ a_{23}a_{32}<1,\ 1+a_{31}>0,\ 1+a_{32}>0,\\
1-a_{12}a_{21}+a_{31}(1+a_{12})+a_{32}(1+a_{21})>0,\
1-a_{23}a_{32}+a_{12}(1+a_{23})+a_{13}(1+a_{32})<0.
\end{array}
$$

A matrix of the local map $H_1^{(in)}\to H_2^{(in)}$ is
$$\left(
\renewcommand{\arraystretch}{1.2}
\begin{array}{cc}
-\mu_{15}/\mu_{13} & 0\\
-\mu_{14}/\mu_{13} & 1
\end{array}
\right),$$
where the basis in $H_1^{(in)}$ is $({\bf e}_3,{\bf e}_4)$ (expanding and
transverse eigenvectors of $df(\xi_1^*)$) and the basis in
$H_2^{(in)}$ is $({\bf e}_5,{\bf e}_4)$ (transverse and
expanding eigenvectors of $df(\xi_2^*)$). Permuting ${\bf e}_4$ and
${\bf e}_5$ in $H_2^{(in)}$ and taking into account the symmetry $\gamma$
we obtain the transition matrix of the cycle
$$\left(
\renewcommand{\arraystretch}{1.2}
\begin{array}{cc}
-\mu_{14}/\mu_{13} & 1\\
-\mu_{15}/\mu_{13} & 0
\end{array}
\right).$$
Since $\mu_{13}$ is positive and other eigenvalues are negative, the
determinant of the
matrix in negative. Therefore, the eigenvalues of this matrix are real
of different signs. Applying the formula for the roots of a quadratic
equation we obtain that the positive root is larger than one whenever
$$-\mu_{14}/\mu_{13}-\mu_{15}/\mu_{13}>1.$$

The values of coefficients of (\ref{lvsys}) employed in numerical simulations are
\begin{equation}\label{exa13}
a_{12}=a_{14}=-0.5,\ a_{13}=-2\hbox{ and }a_{15}=0.5.
\end{equation}
Symmetry $\gamma$ implies that
$$a_{21}=a_{32}=a_{15},\ a_{31}=a_{42}=a_{14},\
a_{41}=a_{52}=a_{13},\ a_{51}=a_{12}.$$

Hence, from (\ref{exa10})-(\ref{exa13})
$$x_1^*=0.4,\ x_2^*=1.2,\ \mu_{13}=1.4,\ \mu_{14}=-0.4\hbox{ and }
\mu_{15}=-1.6.$$
The transition matrix of the cycle is
$$\left(
\renewcommand{\arraystretch}{1.2}
\begin{array}{cc}
0.29 & 1\\
1.22 & 0
\end{array}
\right).$$
The dominant eigenvalue is $\lambda^{max}=1.22>1$, hence by theorem
\ref{th1} the cycle is asymptotically stable.
The results of numerical simulations, a trajectory approaching the homoclinic
cycle, are shown in figure \ref{fig2}.

\begin{figure}
{\large
\hspace*{-12mm}\includegraphics[width=102mm]{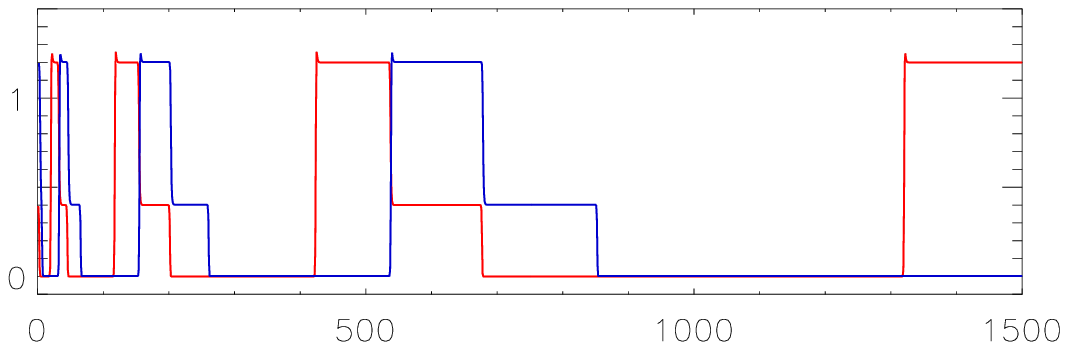}\hspace*{-12mm}\includegraphics[width=102mm]{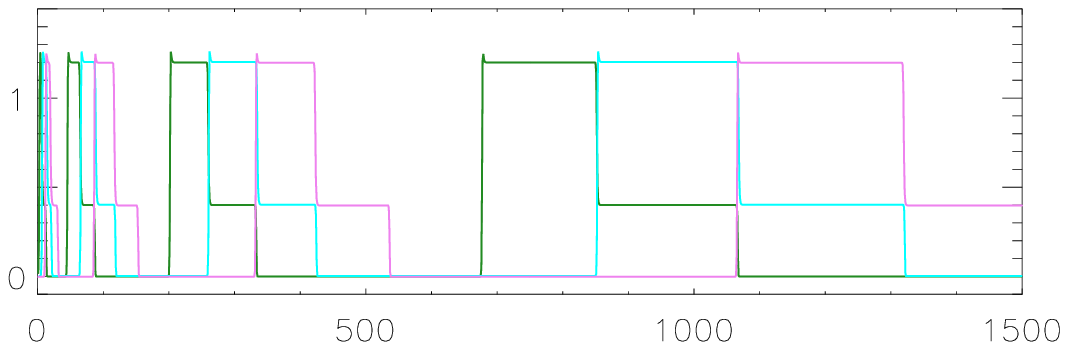}

\vspace*{-33mm}\hspace*{-4mm}$\bf x$

\vspace*{18mm}\hspace*{73mm}$\tau$\hspace*{88mm}$\tau$

\vspace{1mm}

}

\caption{Time dependence of $\bf x$ for the system of section \ref{exa1}:
$x_1$ -- red line, $x_2$ -- blue, $x_3$ -- green, $x_4$ -- cyan, $x_5$ -- violet.
\label{fig2}}
\end{figure}

\subsection{Example 2}\label{exa2}

The cycle of this example is $\xi_1\to\xi_2^*\to\xi_3\to\xi_4\to\xi_5\to\xi_1$,
where the equilibrium $\xi_2^*=(x_1^*,x_2^*,0,0,0)$ belongs to the plane
spanned by ${\bf e}_1$ and ${\bf e}_2$ and other equilibria to the
respective coordinate axes. The connection $\xi_1\to\xi_2^*$ belongs to the
plane $(x_1,x_2,0,0,0)$, the one  $\xi^*_2\to\xi_3$ to the subspace
$(x_1,x_2,x_3,0,0)$, and connections $\xi_j\to\xi_{j+1}$ ($\xi_1\equiv\xi_6$
is assumed) to the planes spanned by ${\bf e}_j$ and ${\bf e}_{j+1}$.

The conditions for the existence of $\xi_2^*$ and the heteroclinic
trajectory $\xi_1\to\xi_2^*$ are that
\begin{equation}\label{exi21}
\lambda_{12}=1+a_{21}>0,\ \lambda_{21}=1+a_{12}>0,\ 1-a_{12}a_{21}>0.
\end{equation}
The trajectory $\xi_j\to\xi_{j+1}$ exists whenever
\begin{equation}\label{exi22}
\lambda_{j,j+1}=1+a_{j+1,j}>0,\ \lambda_{j+1,j}=1+a_{j,j+1}<0,
\end{equation}
see theorem \ref{tlv2}. The coordinates of $\xi_2^*$ are given by (\ref{exa10})
and the eigenvalues of $df(\xi^*_2)$ are
$$\mu_{23}=1+x^*_1a_{31}+x^*_2a_{32},\
\mu_{24}=1+x^*_1a_{41}+x^*_2a_{42},\ \mu_{25}=1+x^*_1a_{51}+x^*_2a_{52}.$$

As proven in appendix B, a sufficient condition for the existence of the
connection $\xi_2^*\to\xi_3$ in $(x_1,x_2,x_3,0,0)$ is that
\begin{equation}\label{exi221}
\begin{array}{l}
1+a_{12}>0,\ 1+a_{13}<0,\ 1+a_{21}>0,\ 1+a_{23}<0,\
1+a_{31}>0,\ 1+a_{32}>0,\\
a_{12}a_{21}>1,\
1-a_{12}a_{21}+a_{31}(1+a_{12})+a_{32}(1+a_{21})>0.
\end{array}
\end{equation}

Let the coordinates in $H_j^{(in)}$ be chosen as follows:
$$H_1^{(in)}:(x_2,x_3,x_4),\ H_2^{(in)}:(x_3,x_4,x_5),\ H_3^{(in)}:(x_4,x_5),\
H_4^{(in)}:(x_1,x_2,x_5),\ H_5^{(in)}:(x_1,x_2,x_3).$$
From (\ref{esm}), in these coordinates the basic transition matrices
$M_j:H_j^{(in)}\to H_{j+1}^{(in)}$ are
\begin{equation}\label{basm1}
M_1=\left(
\renewcommand{\arraystretch}{1.2}
\begin{array}{ccc}
-\lambda_{13}/\lambda_{12} & 1 & 0 \\
-\lambda_{14}/\lambda_{12} & 0 & 1 \\
-\lambda_{15}/\lambda_{12} & 0 & 0
\end{array}
\right),\
M_2=\left(
\renewcommand{\arraystretch}{1.2}
\begin{array}{ccc}
-\mu_{24}/\mu_{23} & 1 & 0 \\
-\mu_{25}/\mu_{23} & 0 & 1
\end{array}
\right),\
M_3=\left(
\renewcommand{\arraystretch}{1.2}
\begin{array}{cc}
-\lambda_{31}/\lambda_{34} & 0 \\
-\lambda_{32}/\lambda_{34} & 0 \\
-\lambda_{35}/\lambda_{34} & 1
\end{array}
\right),
\end{equation}
\begin{equation}\label{basm2}
M_4=\left(
\renewcommand{\arraystretch}{1.2}
\begin{array}{ccc}
1 & 0 & -\lambda_{41}/\lambda_{45} \\
0 & 1 & -\lambda_{42}/\lambda_{45} \\
0 & 0 & -\lambda_{43}/\lambda_{45}
\end{array}
\right),\
M_5=\left(
\renewcommand{\arraystretch}{1.2}
\begin{array}{ccc}
-\lambda_{52}/\lambda_{51} & 1 & 0 \\
-\lambda_{53}/\lambda_{51} & 0 & 1 \\
-\lambda_{54}/\lambda_{51} & 0 & 0
\end{array}
\right).
\end{equation}
We prescribe that all transverse eigenvalues are negative, except for
$\lambda_{13}$. Namely, that
\begin{equation}\label{exi23}
\lambda_{14},\lambda_{35},\lambda_{41},\lambda_{42},\lambda_{52},\lambda_{53},\mu_{24},\mu_{25}<0.
\end{equation}

The following values of coefficients of the system (\ref{lvsys}) were employed
in computations:
$$
\begin{array}{l}
a_{13}=a_{23}=a_{34}=b_{45}=-3,\
a_{43}=a_{54}=a_{15}=1,\ a_{12}=a_{21}=0.5,\\
a_{31}=-0.75,\ a_{32}=0.75,\hbox{ other }a_{ij}=-1.02\hbox{ for }i\ne j.
\end{array}
$$
The coefficients satisfy the conditions (\ref{exi21})-(\ref{exi221}) for
the existence of the equilibrium $\xi^*_2$ and heteroclinic connections,
and that the transverse eigenvalues (\ref{exi23}) are negative.

For these values of parameters the $2\times 2$ transition matrix
$M^{(3)}=M_2M_1M_5M_4M_3$ is
$$
M^{(3)}\approx\left(
\renewcommand{\arraystretch}{1.2}
\begin{array}{cc}
0.5 & 3 \\
-0.5 & 3
\end{array}
\right).
$$
The dominant eigenvalue is $\lambda_{max}\approx 2$ and the associated
eigenvector is $v_{max}\approx (2,1)$. The eigenvector of $M^{(2)}$ associated
with $\lambda_{max}$ is $M_1M_5M_4M_3(2,1)\approx(0.7,2,0.03)$.

Hence, $\lambda_{max}$ and associated eigenvector satisfy the conditions
of theorem \ref{th2} for fragmentary asymptotic stability of the cycle.
The results of numerical simulations of a trajectory approaching the
cycle are shown in figure \ref{fig3}.

\begin{figure}
{\large
\hspace*{-12mm}\includegraphics[width=102mm]{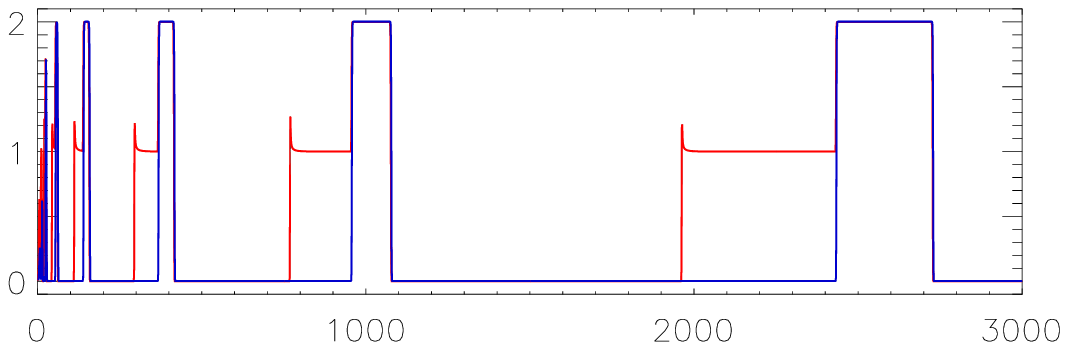}\hspace*{-12mm}\includegraphics[width=102mm]{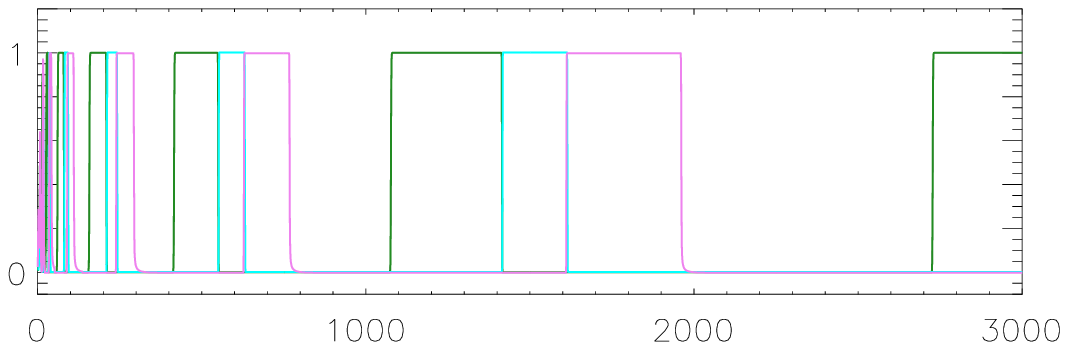}

\vspace*{-33mm}\hspace*{-4mm}$\bf x$

\vspace*{18mm}\hspace*{73mm}$\tau$\hspace*{88mm}$\tau$

\vspace{1mm}

}

\caption{Time dependence of $\bf x$ for the system of section \ref{exa2}:
$x_1$ -- red line, $x_2$ -- blue, $x_3$ -- green, $x_4$ -- cyan, $x_5$ -- violet.
\label{fig3}}
\end{figure}

\section{Conclusion}\label{sec_conc}

We have defined type Y heteroclinic cycles, that include
type Z heteroclinic cycles \cite{op12}, quasi-simple cycles
of \cite{gar} and
heteroclinic cycles in pluridimensions studied in \cite{cas25}.
We have proved necessary and sufficient conditions
for asymptotic stability or fragmentary asymptotic stability of the
cycles under the assumption that the contracting eigenvalues are negative
and the radial eigenvalues have negative real parts.
The conditions involve the eigenvalues and eigenvectors of
transition matrices, as this is the case for type Z cycles.
The matrices are products of basic transition matrices that depend
on eigenvalues of the linearisations near steady states comprising the cycle
and the dimension of the contracting subspace. The definition of
type Y cycles can be extended to type Y omnicycles \cite{op23b},
that differ from cycles in that the equilibria
in the sequence $\{\xi_1,\xi_2,...,\xi_n\}$ are not required to be distinct.
The conditions for fragmentary asymptotic stability of type Y heteroclinic
cycles hold true for type Y  omnicycles. (Omnicycles, that are not
heteroclinic cycles, are not asymptotically stable.)

We discuss two examples of type Y heteroclinic cycles admitted by the
generalised Lotka-Volterra system (GLV) in $\R^5$, one of the cycles is
asymptotically stable, the other one is fragmentarily asymptotically stable.
More examples of type Y heteroclinic cycles in the GLV system in $\R^n$
can be constructed similarly by using the conditions for the existence
of heteroclinic trajectories in two- or three-dimensional subspaces given
in appendix B or in \cite{op23a}. The results on the existence of heteroclinic
connections can be used for the construction of other types of cycles and
omnicycles, possibly with two-dimensional connections, and heteroclinic
networks.

Unlike in earlier studies of asymptotic
and/or fragmentary asymptotic stability of heteroclinic cycles we do not
require the eigenvalues of linearisations near equilibria to be distinct, the cycle to be robust
and the flow-invariant subspaces that heteroclinic trajectories
belong to be of equal dimension. Nevertheless, the conditions for
asymptotic stability of type Z cycles hold true for type Y as well.
Natural question arising here is how large is the class of heteroclinic cycles
where the conditions for stability are applicable. In particular, what
happens if we allow the radial or contracting eigenvalue to be positive,
or the expanding subspace to be of dimension larger than one.

Another possible continuation of the study is the investigation of bifurcations
occurring when the cycle ceases to be (fragmentarily) asymptotically
stable or the cycle is destroyed. Some of these bifurcations for type Z
cycles were studied in \cite{op12}, it might be the case that the results
hold true for type Y cycles.

\appendix

\section{The rank of transition matrices for type Y heteroclinic
cycles}\label{rankM}

As defined in section \ref{smaps}, a transition matrix of the collection of
maps associated with a heteroclinic cycle is a product of basic transition
matrices,
\begin{equation}\label{trMa}
M^{(j)}=M_{j+1}...M_mM_1...M_j.
\end{equation}
A basic matrix $M_i:H_i\to H_{i+1}$ has the dimension $n_{i+1}\times n_i$ and
$M^{(j)}$ is a square $n_j\times n_j$ matrix. Let the
equilibria be ordered in such a way that $n_1\le n_i$ for any
$2\le i\le m$. If $\lambda$ is an eigenvalue of $M^{(1)}$
with the associated eigenvector $v$ then $\lambda$ also is an eigenvalue of
$M^{(j)}$ with the associated eigenvector $M_{j+1}...M_mv$. The other
$n_j-n_1$ eigenvalues of $M^{(j)}$ vanish. Hence, the rank of any matrix
$M^{(j)}$ equals to the rank of $M^{(1)}$. Since $M^{(1)}$ is a product
of matrices of various dimensions where most of the entries are zeros,
it is not evident that $M^{(1)}$ is not
degenerate. In this appendix we prove that generically $\rank M^{(1)}=n_1$,
which is equivalent to that $\det M^{(1)}\ne 0$.

The proof is based on the following proposition and lemma:
\begin{proposition}\label{prm}\cite{mit}
Let $A(x)$ be a real analytic function on (a connected open domain $U$ of)
$\R^d$. If $A$ is not identically zero, then its zero set
$$F(A) :=\{\ x\in U\ :\  A(x)=0\ \}$$
has a zero measure.
\end{proposition}

\begin{lemma}\label{prodm}
Let $C=AB$, where $A$ is an $m\times n$ matrix and $B$ is an $n\times p$
matrix. Suppose that all entries in the lower left $(m-r)\times r$
corner of the matrix $A$ vanish and entries in the in the lower left
$(n-r)\times r$ corner of the matrix $B$ also vanish for some
$0\le r\le \min(m,n,p)$. Then all entries the lower left $(m-r)\times r$
corner of the $m\times p$ matrix $C$ vanish and the
upper left $r\times r$ corner of the matrix $C$ is the product of the
respective corners of the matrices $A$ and $B$.
\end{lemma}
The proof of the lemma follows from the definition of the matrix product.

\medskip
A basic transition matrix has the form (\ref{esm})
\begin{equation}\label{esmA}
M_j:=A_j\left(
\begin{array}{ccccc}
b_{j,1}&0&0&\ldots&0\\
.&.&.&\ldots&.\\
b_{j,n_j^c}&0&0&\ldots&0\\
b_{j,n_j^c+1}&1&0&\ldots&0\\
b_{j,n_j^c+2}&0&1&\ldots&0\\
.&.&.&\ldots&.\\
b_{j,n_{j+1}}&0&0&\ldots&1
\end{array}
\right),
\end{equation}
where $n_{j+1}=n^t_{j+1}+1=n^c_j+n^t_j$, $n_{j+1}\ge n_j-1$ and $A_j$
is a permutation matrix.

The determinant of $M^{(1)}$ can be regarded as a function of
${\bf b}=\{b_{j,i}\}$, $1\le i\le n_j,1\le j\le m$:
$$F({\bf b})\equiv\det M^{(1)}=\det(M_m...M_1).$$
Being a polynomial function, $F({\bf b})$ is an analytic one.
In agreement with proposition \ref{prm}, to prove that the function is not
identically zero, we aim on finding
a set of variables ${\bf b}$ such that $F({\bf b})\ne 0$.

To do this, we split the vectors $\{v_i^j\}_{i=1}^{n_j}$ comprising
the bases in $H_j$, $1\le j\le m$, into two groups, the first
group and the second one. The splitting is coded by a vector
${\bf h}^j=(h_1^j,...,h_{n_j}^j)$, where $h_i^j=1$ if $v_i^j$ belongs to the
first group and $h_i^j=0$ otherwise. The splitting
is performed by setting ${\bf h}^1={\bf 1}$ and employing the following
recurrent procedure:
\begin{itemize}
\item[case I] $\hbox{If }n_{j+1}=n_j\hbox{ then }{\bf h}^{j+1}=A_j{\bf h}^j.$
\item[case II] $\hbox{If }n_{j+1}-n_j=k\hbox{ with }k>0\hbox{ then }
{\bf h}^{j+1}=A_j\tilde{\bf h}\hbox{ where }
\tilde{\bf h}=(0,...,0,h^j_1,h^j_2,...,h_{n_j}^j).$
\item[case III] $\hbox{If }n_{j+1}=n_j-1\hbox{ and }h^j_1=0\hbox{ then }
{\bf h}^{j+1}=A_j\tilde{\bf h}\hbox{ where }
\tilde{\bf h}=(h^j_2,...,h_{n_j}^j).$
\item[case IV]$\hbox{If }n_{j+1}=n_j-1\hbox{ and }h^j_1=1\hbox{ then }
{\bf h}^{j+1}=A_j\tilde{\bf h}\hbox{ where }
\tilde{\bf h}=(h^j_2,.,h^j_{s-1},1,h^j_{s+1},..,h_{n_j}^j)\break
\hbox{ and }h_s^j=0.$
\end{itemize}
By construction, any of ${\bf h}^j$ has exactly $n_1$ ones.
The values of $\{b_{ij}\}$ are chosen as follows:
\begin{itemize}
\item[case I] $b_{j1}=1$ and $b_{ji}=0$ for other $i$.
\item[case II] $b_{j,k+1}=1$ and $b_{ji}=0$ for other $i$.
\item[cases III] $b_{ji}=0$ for all $2\le i\le n_j$.
\item[case IV] $b_{js}=1$ and
$b_{ji}=0$ for other $i$.
\end{itemize}

Denote by $M_j^*$ the basic transition matrices in the new bases, which
are permutations of the old bases with vectors in the first group going
first. Due to our choice of $\{b_{ij}\}$, the upper left $n_1\times n_1$
corner of $M_j^*$ for any $j$ is a permutation matrix. The entries
in the lower left $(n_j-n_1)\times n_1$ corner vanish.
Applying $m-1$ times lemma \ref{prodm} to the product $M^{(1)}=M_m^*...M_1^*$
starting from $M_2^*M_1^*$ we obtain that $M^{(1)}$ is a permutation
matrix and $\det M^{(1)}=\pm1\ne0$. Hence,
we have identified the values of $\bf b$ such that $F({\bf b})\ne0$.
Therefore, by proposition \ref{prm} generically the matrix $M^{(1)}$
is not degenerate.

\section{Heteroclinic connections in the Lotka-Volterra system in $\R^3_+$}\label{lvol}

Before proving the conditions for the existence of heteroclinic
connections we recall theorems (see, e.g., \cite{hs98,op23a}) that will be
used in the proofs. We denote
$$\R^n_+=\{~\bx\in\R^n~:~x_j>0,\ j=1,...,n~\}\hbox{ and }
\R^n_{0,+}=\{~\bx\in\R^n~:~x_j\ge0,\ j=1,...,n~\}$$
and consider the generalised Lotka-Volterra (GLV) system in $\R^n_{0,+}$
\begin{equation}\label{sysap}
\dot x_i=b_j(r_i+\sum_{j=1}^n a_{ij}x_j),\quad i=1,...,n.
\end{equation}
A steady state of (\ref{sysap}), $\bx=(x_1,...,x_n)\in\R^n_{0,+}$, is called
{\it interior} if $x_1...x_n\ne0$.

\begin{theorem}\label{intp1} The interior of $\R^n_{0,+}$ contains $\alpha-$ and
$\omega-$ limit points if and only if (\ref{lvsys}) admits an interior steady
state.
\end{theorem}

\begin{theorem}\label{tlv2} Consider the system
\begin{equation}\label{glv2}
\renewcommand{\arraystretch}{1.3}
\begin{array}{l}
\dot x_1=\alpha_1x_1(1-x_1+\beta_1x_2)\\
\dot x_2=\alpha_2x_2(1-x_2+\gamma_2x_1),
\end{array}
\end{equation}
where
\begin{equation}\label{glv2c}
\alpha_1>0\hbox{ and }\alpha_2>0.
\end{equation}
Let $\beta_1$ and $\gamma_2$ satisfy one of the following sets of inequalities
\begin{equation}\label{uneq}
\renewcommand{\arraystretch}{1.4}
\begin{array}{ll}
(a) & 1+\beta_1<0,\ 1+\gamma_2>0\\
(b) & 1+\beta_1>0,\ 1+\gamma_2<0\\
(c) & 1+\beta_1>0,\ 1+\gamma_2>0,\ \beta_1\gamma_2<1\\
(d) & 1+\beta_1>0,\ 1+\gamma_2>0,\ \beta_1\gamma_2>1\\
(e) & 1+\beta_1<0,\ 1+\gamma_2<0.
\end{array}
\end{equation}
then the system has steady states $\xi_1=(1,0)$ and $\xi_2=(0,1)$ stable
along the respective axes. In case (a) there exists a heteroclinic trajectory
$\xi_1\to\xi_2$ and all trajectories in $\R^2_+$ are
attracted by $\xi_2$; (b) there exists a heteroclinic trajectory
$\xi_2\to\xi_1$ and all trajectories in $\R^2_+$ are attracted by $\xi_2$;
(c) there exists a steady state $\xi^*$
\begin{equation}\label{ss4}
\xi^*=\biggl({1+\beta_1\over1-\beta_1\gamma_2},
{1+\gamma_2\over1-\beta_1\gamma_2}\biggr).
\end{equation}
and heteroclinic trajectories $\xi_1\to\xi^*$ and $\xi_2\to\xi^*$, and
all trajectories in $\R^2_+$ are attracted by stable $\xi^*$;
(d) there are no steady states, except for $\xi_0=(0,0)$, $\xi_1$ and $\xi_2$,
which are unstable and for $t\to\infty$ all trajectories in $\R^2_+$ go to
$(\infty,\infty)$; (e) there exists a steady state $\xi^*$ (\ref{ss4})
and heteroclinic trajectories $\xi^*\to\xi_1$ and $\xi^*\to\xi_2$ and
almost all trajectories in $\R^2_+$ (except for the one-dimensional stable
manifold of $\xi^*$) are attracted either by stable
$\xi_1$ or by stable $\xi_2$.
\end{theorem}

Consider the GLV system in $\R^3_{0,+}$ and re-write (\ref{sysap}) as
\begin{equation}\label{glv3}
\begin{array}{l}
\dot x_1=\alpha_1x_1(1-x_1+\beta_1x_2+\gamma_1x_3)\\
\dot x_2=\alpha_2x_2(1-x_2+\beta_2x_3+\gamma_2x_1)\\
\dot x_3=\alpha_3x_3(1-x_3+\beta_3x_1+\gamma_3x_2).
\end{array}
\end{equation}
For any values of the coefficients $\alpha_i$, $\beta_i$ and $\gamma_i$, the
system has four steady states: the origin $\xi_0=(0,0,0)$ and one equilibrium on
each of the coordinate axes, $\xi_1=(1,0,0)$, $\xi_2=(0,1,0)$ and $\xi_3=(0,0,1)$.
We assume that the equilibria on coordinate axes are stable along the respective
directions, which is the case if
\begin{equation}\label{cond1}
\alpha_i>0\hbox{ for }i=1,2,3.
\end{equation}

\begin{theorem}\label{tlv30} Consider the system (\ref{glv3}),(\ref{cond1}). If
\begin{equation}\label{uneq3}
\renewcommand{\arraystretch}{1.4}
\begin{array}{l}
-1<\beta_1<0,\ 1+\gamma_1<0,\ 1+\beta_2>0,\ 1+\gamma_2>0\\
1+\beta_3>0,\ 1+\gamma_3>0,\ \beta_2\gamma_3<1,\ \beta_1\gamma_2<1\\
1-\beta_2\gamma_3+\beta_1(1+\beta_2)+\gamma_1(1+\gamma_3)<0\\
1-\beta_1\gamma_2+\beta_3(1+\beta_1)+\gamma_3(1+\gamma_2)>0
\end{array}
\end{equation}
then the phase portrait of the system is the respective one shown in figure
\ref{fig1}a. Namely, there exist steady states $\xi^*_1$ and $\xi^*_2$ in the
planes
$(x_1,x_2,0)$ and $(0,x_2,x_3)$, respectively, and heteroclinic trajectories
$\xi_1\to\xi_3$, $\xi_1\to\xi_1^*$, $\xi_2\to\xi_1^*$, $\xi_2\to\xi_2^*$,
$\xi_3\to\xi_2^*$ and $\xi_1^*\to\xi_2^*$, and all trajectories in $\R^3_+$
are attracted by the stable $\xi_2^*$.
\end{theorem}

\proof
The existence of steady state $\xi_1^*$ and $\xi_2^*$,
and heteroclinic trajectories $\xi_1\to\xi_3$,
$\xi_1\to\xi_1^*$, $\xi_2\to\xi_1^*$, $\xi_2\to\xi_2^*$ and $\xi_3\to\xi_2^*$
follows from theorem \ref{tlv2}.

The interior steady state $(\tilde x_1,\tilde x_2,\tilde x_3)$ is the solution
of the linear system
\begin{equation}\label{lins}
\renewcommand{\arraystretch}{1.2}
\begin{array}{l}
x_1-\beta_1x_2-\gamma_1x_3=1\\
x_2-\beta_2x_3-\gamma_2x_1=1\\
x_3-\beta_3x_1-\gamma_3x_2=1
\end{array}
\end{equation}
By the Cramer's rule
\begin{equation}\label{crule}
\renewcommand{\arraystretch}{1.2}
\begin{array}{l}
\tilde x_1=(1-\beta_2\gamma_3+\beta_1(1+\beta_2)+\gamma_1(1+\gamma_3))/\det A\\
\tilde x_2=(1-\beta_3\gamma_2+\beta_2(1+\beta_3)+\gamma_2(1+\gamma_2))/\det A\\
\tilde x_3=(1-\beta_1\gamma_2+\beta_3(1+\beta_1)+\gamma_3(1+\gamma_2))/\det A
\end{array}
\end{equation}
where $A$ is the matrix of linear system (\ref{lins}). The conditions of
the theorem imply that $\tilde x_1$ and $\tilde x_3$ have different signs,
therefore the system (\ref{glv3}) does not have steady states in $\R^3_+$.

Consider the box $\widehat U$,
$$\widehat U=\{~~(x_1,x_2,x_3)~:~0\le x_1\le \hat x_1,\ 0\le x_2\le \hat x_2,\
0\le x_3\le \hat x_3~~\},$$
where the values of $\hat x_i$, $i=2,3$, are different for different signs
of $\beta_2$ and $\gamma_3$. Namely, we set $\hat x_1\ge 1$ and
\begin{equation}\label{box}
\renewcommand{\arraystretch}{1.2}
\begin{array}{lcl}
\beta_2\le0,\ \gamma_3\le0&:
& \hat x_2\ge 1+\hat x_1s(\gamma_2),\ \hat x_3\ge 1+\hat x_1s(\beta_3)\\
\beta_2\le0,\ \gamma_3>0&:
& \hat x_2\ge 1+\hat x_1s(\gamma_2),\ \hat x_3\ge 1+\hat x_1s(\beta_3)+\hat x_2\gamma_3\\
\beta_2>0,\ \gamma_3\le0&:
& \hat x_3\ge 1+\hat x_1s(\beta_3),\
\hat x_2\ge 1+\hat x_1s(\gamma_2)+\hat x_3\beta_2\\
\beta_2>0,\ \gamma_3>0 &:&
 q\ge 1+\hat x_1s(\beta_3)+\hat x_1s(\gamma_2),\
\hat x_2=q(1+\beta_2)/(1-\beta_2\gamma_3),\\
&&\hat x_3=q(1+\gamma_3)/(1-\beta_2\gamma_3),
\end{array}
\end{equation}
where $s(y)=\max(0,y)$. The inequalities together with
the conditions of the theorem imply that $\dot x_j<0$ at the $x_j=\hat x_j$
side of the box $\widehat U$. Since $\hat x_j$ can be taken arbitrary large,
together with the invariance of the coordinate
planes this implies that any trajectory in $\R^3_+$ enters the forward
invariant box $\widehat U$ for some $\tau>0$. By theorem \ref{intp1}
the system (\ref{glv3}) does not have $\omega$-sets inside the box,
therefore all trajectories in $\R^3_+$ are attracted by the stable $\xi_2^*$.
\qed

\begin{theorem}\label{tlv3} Consider the system (\ref{glv3}),(\ref{cond1}). If
\begin{equation}\label{uneq30}
\renewcommand{\arraystretch}{1.4}
\begin{array}{l}
1+\beta_1>0,\ 1+\gamma_1<0,\ 1+\beta_2<0,\ 1+\gamma_2>0,
\ 1+\beta_3>0,\ 1+\gamma_3>0\\
\beta_1\gamma_2<1,\
1-\beta_1\gamma_2+\beta_3(1+\beta_1)+\gamma_3(1+\gamma_2)>0
\end{array}
\end{equation}
then the phase portrait of the system is the respective one shown in figure
\ref{fig1}b. Namely, there exist steady state $\xi_2^*$ in the plane
$(x_1,x_2,0)$, heteroclinic trajectories $\xi_1\to\xi_2^*$,
$\xi_1\to\xi_3$, $\xi_2\to\xi_2^*$, $\xi_2\to\xi_3$ and $\xi_2^*\to\xi_3$,
and all trajectories in $\R^3_+$ are attracted by stable $\xi_3$.
\end{theorem}

\proof
We use the notation $N_j$ for the sets
\begin{equation}\label{nset}
\renewcommand{\arraystretch}{1.2}
\begin{array}{l}
N_1=\{~\bx\in\R^3_{0,+}~:~1-x_1+\beta_1x_2+\gamma_1x_3=0~\},\\
N_2=\{~\bx\in\R^3_{0,+}~:~1-x_2+\beta_2x_3+\gamma_2x_1=0~\},\\
N_3=\{~\bx\in\R^3_{0,+}~:~1-x_3+\beta_3x_1+\gamma_3x_2=0~\}\\
\end{array}
\end{equation}
and denote $L_{ij}=N_i\cap N_j$.

The existence of the steady state $\xi_2^*$ in $(x_1,x_2,0)$ and
heteroclinic trajectories $\xi_1\to\xi_2^*$,
$\xi_1\to\xi_3$, $\xi_2\to\xi_2^*$ and $\xi_2\to\xi_3$ follows from
theorem \ref{tlv2}.

The sets $N_1$ and $N_2$ intersect the coordinate axis $x_3$ at the
points $(0,0,q_{01})$ and $s^*_0=(0,0,q_{02})$, respectively, where
$q_{01}=-1/\gamma_1<1$ and $q_{02}=-1/\beta_2<1$ by the conditions of the
theorem. For definiteness, assume that $q_{02}>q_{01}$. Let
$(1,0,q_{11})$ and $s^*_1=(1,0,q_{12})$ be the intersections of the line
$(1,0,x_3)$ with the sets $N_1$ and $N_2$, respectively (see figure \ref{fig5}).
We have $q_{11}=-(1+\beta_1)/\gamma_1>0$ and $q_{12}=0$ , i.e.,
$q_{11}>q_{12}$. Therefore, the point of intersection of $L_{12}$ with
the plane $x_1=0$, that we denote by $s^*$, belongs to the line segment
bounded by $s^*_0$ and $s^*_1$. The conditions of the theorem imply that
$\dot x_3(s^*_0)>0$ and $\dot x_3(s^*_1)>0$, therefore $\dot x_3(s^*)>0$.

\begin{figure}
{\large
\vspace{-2cm}\hspace*{-12mm}\includegraphics[width=90mm]{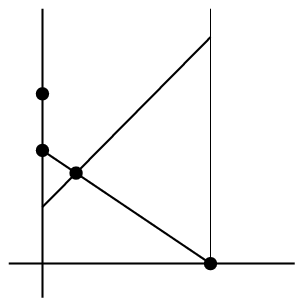}

\vspace*{-34mm}\hspace*{59mm}$x_2$

\vspace*{-43mm}\hspace*{12mm}$x_3$

\vspace*{25mm}\hspace*{24mm}$s^*$

\vspace*{-15mm}\hspace*{14mm}$s^*_0$

\vspace*{-15mm}\hspace*{14mm}$\xi_3$

\vspace*{24mm}\hspace*{42mm}$s^*_1=\xi_2$

\vspace*{-42mm}\hspace*{35mm}$N_1$

\vspace*{16mm}\hspace*{36mm}$N_2$

\vspace{12mm}

}

\caption{
Intersection of the sets $N_1$ and $N_2$ with the plane
$(0,x_2,x_3)$.
\label{fig5}}
\end{figure}

The octant $\R^3_{0,+}$ is decomposed by $N_3$ into two subsets, $U_+$ and $U_-$,
such that $\dot x_3(\bx)>0$ for $\bx\in U_+$ and $\dot x_3(\bx)\le0$ for
$\bx\in U_-$. The set $L_{12}$ is a line segment, bounded by
$\xi_2^*$ from below and by $s^*$ from above. By the conditions of the
theorem $\dot x_3(\xi_2^*)>0$, which implies that both $\xi_2^*$ and $s^*$
belong to $U_+$. Hence, $L_{12}$ does not intersect $N_3$. Therefore,
the system (\ref{uneq30}) does not have an interior steady state in $\R^3_+$.

Consider the box $\widehat U$,
$$\widehat U=\{~~(x_1,x_2,x_3)~:~0\le x_1\le \hat x_1,\ 0\le x_2\le \hat x_2,\
0\le x_3\le \hat x_3~~\},$$
where the values of $\hat x_i$, $i=1,2$, are different for different signs
of $\gamma_2$ and $\beta_3$,namely
\begin{equation}
\renewcommand{\arraystretch}{1.2}
\begin{array}{lcl}
\beta_1\le0,\ \gamma_2\le0&:
& \hat x_1\ge 1,\ \hat x_2\ge 1\\
\beta_1\le0,\ \gamma_2>0&:
& \hat x_1\ge 1,\ \hat x_2\ge 1+\hat x_1\gamma_2\\
\beta_1>0,\ \gamma_2\le0&:
& \hat x_2\ge 1,\ \hat x_1\ge 1+\hat x_2\beta_1\\
\beta_1>0,\ \gamma_2>0 &:&
 q\ge 1,\
\hat x_1=q(1+\beta_1)/(1-\beta_1\gamma_2),\\
&&\hat x_2=q(1+\gamma_2)/(1-\beta_1\gamma_2),
\end{array}
\end{equation}
and $\hat x_3>1+\beta_3\hat x_1+\gamma_3\hat x_2$.
The proof that $\widehat U$ is forward invariant
is similar to the one in theorem \ref{tlv30}.
By theorem \ref{intp1} the system (\ref{glv3}) does not have $\omega$-sets
inside the box, therefore all trajectories are attracted by the stable $\xi_2^*$.
\qed

\begin{remark}\label{rem2}
Theorem \ref{tlv3} implies existence of two-dimensional heteroclinic connections
$\xi_1\to\xi_2^*$ and $\xi_2\to\xi_2^*$, while by theorem \ref{tlv30} the
connections $\xi_1\to\xi_3$ and $\xi_2\to\xi_3$ are two-dimensional.
These can be used for the construction of examples of dynamical
system possessing heteroclinic cycles or networks involving two-dimensional
connections.
\end{remark}


\begin{thebibliography}{99}

\bibitem{ash24}
P.~Ashwin, M.~Fadera, C.~Postlethwaite.
Network attractors and nonlinear dynamics of neural computation.
{\it Current Opinion in Neurobiology} {\bf 84}, 102818 (2024).

\bibitem{agh88}
D.~Armbruster, J.~Guckenheimer and P.~Holmes.
\newblock Heteroclinic cycles and modulated travelling
waves in systems with $O(2)$ symmetry.
\newblock {\it Physica D} {\bf 29}, 257 -- 282 (1988).



\bibitem{ac10}
M.A.D.~Aguiar and S.B.S.D.~Castro.
\newblock Chaotic switching in a two-person game.
\newblock {\it Physica D} {\bf 239}, 1598 -- 1609 (2010).

\bibitem{bic19}
C.~Bick, A.~Lohse.
Heteroclinic Dynamics of Localized Frequency Synchrony:
Stability of Heteroclinic Cycles and Networks.
{\it J. Nonlinear Sci.} {\bf 29}, 2571--2600, (2019).

\bibitem{bra94}
W.~Brannath.
\newblock Heteroclinic networks on the tetrahedron.
\newblock {\em Nonlinearity} {\bf 7}, 1367 -- 1384 (1994).

\bibitem{cas22}
S.~B.~S.~D.~Castro, A.~Ferreira, L.~Garrido-da-Silva, I.~S.~Labouriau.
Stability of Cycles in a Game of Rock-Scissors-Paper-Lizard-Spock.
{\it SIAM J. Appl. Dyn. Syst.} {\bf 21}, (2022).

\bibitem{cas25}
S.B.S.D.~Castro, A.M.~Rucklidge.
Robust Heteroclinic Cycles in Pluridimensions.
{\it J. Nonlinear Sci.} {\bf 35}, 80, (2025).

\bibitem{ch03}
P.~Chossat, D.~Armbruster.
Dynamics of polar reversals in spherical dynamos.
{\it Proc. R. Soc. Lond. A} {\bf 459}, 577--596 (2003).

\bibitem{clp}
P.~Chossat, A.~Lohse, O.~Podvigina.
Pseudo-simple heteroclinic cycles in $R^4$.
\newblock {\it Physica D} {\bf 372}, 1 -- 21 (2018).

\bibitem{cro03}
D.T.~Crommelin.
Regime transitions and heteroclinic connections in a barotropic atmosphere.
{\it J. Atmos. Sci.} {\bf 60}, 229--246 (2003).

\bibitem{cru12}
M.~Crucifix.
Oscillators and relaxation phenomena in Pleistocene climate theory.
{\it Phil. Trans. R. Soc. A} {\bf 370}, 1140--1165 (2012).

\bibitem{fi96}
M.~Field.
\newblock {\em Lectures on bifurcations, dynamics and symmetry.}
\newblock Longman: Harlow, England 1996.

\bibitem{hs98}
J.~Hofbauer and K.~Sigmund.
\newblock {\em Evolutionary games and population Dynamics.}
\newblock CUP: Cambridge, 1998.

\bibitem{gar}
L.~Garrido-da-Silva, S.B.S.D.~Castro.
Stability of quasi-simple heteroclinic cycles.
{\it Dynamical systems} {\bf 34}, 14--39 (2019).

\bibitem{gukhol}
J.~Guckenheimer and P.~Holmes.
\newblock {\em Nonlinear oscillations, dynamical systems and bifurcations of
vector fields.} Applied mathematical sciences; vol. 42.
\newblock Springer-Verlag: New York, 1983.

\bibitem{ks94}
V.~Kirk and M. Silber.
\newblock A competition between heteroclinic cycles.
\newblock {\em Nonlinearity} {\bf 7}, 1605 -- 1621 (1994).

\bibitem{Kru97}
M.~Krupa.
\newblock Robust heteroclinic cycles.
\newblock {\em Journal of Nonlinear Science}, {\bf 7}, 129--176 (1997).

\bibitem{km95a}
M.~Krupa and I.~Melbourne.
\newblock Asymptotic stability of heteroclinic cycles in systems with symmetry.
\newblock {\em Ergodic Theory Dyn. Syst.}, {\bf 15}, 121--148 (1995).

\bibitem{km95b}
M.~Krupa and I.~Melbourne.
\newblock Nonasymptotically stable attractors in ${\bf O}(2)$ mode interaction.
\newblock {\em Normal Forms and Homoclinic Chaos} (W.F. Langford and W. Nagata, eds.)
Fields Institute Communications {\bf 4}, Amer. Math. Soc.,
1995, 219--232.

\bibitem{km04}
M.~Krupa and I.~Melbourne.
\newblock Asymptotic stability of heteroclinic cycles in systems with symmetry, II.
\newblock {\em Proc. Roy. Soc. Edinburgh} {\bf 134A}, 1177--1197 (2004).

\bibitem{may75}
R.M.~May, W.J.~Leonard.
Nonlinear aspects of competition between three species,
{\it SIAM J. Appl. Math.} {\bf 29}, 243--253 (1975).

\bibitem{Mel91}
I. Melbourne.
\newblock An example of a non-asymptotically stable attractor.
\newblock {\em Nonlinearity} {\bf 4}, 835 -- 844 (1991).

\bibitem{mit}
B.S.~Mityagin.
The Zero Set of a Real Analytic Function.
{\it Mat. Zametki} {\bf 107}, 473--475 (2020), in Russian.

\bibitem{pf10}
F.~Pailis, S.~Fauve.
Mechanisms for magnetic field reversals.
{\it Phil. Trans. R. Soc. A} {\bf 368}, 1595--1605 (2010).

\bibitem{op12}
O.~Podvigina.
Stability and bifurcations of heteroclinic cycles of type Z.
\newblock {\it Nonlinearity} {\bf 25}, 1887 -- 1917 (2012).

\bibitem{op20}
O.~Podvigina.
Heteroclinic Cycles in Nature.
{\it Izvestiya Physics of the Solid Earth} {\bf 56}, 117--124 (2020).

\bibitem{op23a}
O.~Podvigina.
Two-dimensional heteroclinic connections in the generalized Lotka-Volterra
system.
{\it Dynamical systems} {\bf 38}, 163--178 (2023).

\bibitem{op23b}
O.~Podvigina.
Behaviour of trajectories near a two-cycle heteroclinic network
{\it Dynamical systems} {\bf 38}, 576--596 (2023).

\bibitem{pa11}
O.M.~Podvigina and P.~Ashwin.
\newblock On local attraction properties and a stability index for
heteroclinic connections.
\newblock {\it Nonlinearity} {\bf 24}, 887 -- 929 (2011).

\bibitem{pos10}
C.M.~Postlethwaite.
\newblock A new mechanism for stability loss from a heteroclinic cycle.
\newblock {\it Dynamical systems} {\bf 25}, 305 -- 322 (2010).

\bibitem{pos22a}
C.~M.~Postlethwaite, A.~M.~Rucklidge.
Stability of cycling behaviour near a heteroclinic network model of
Rock-Paper-Scissors-Lizard-Spock.
\newblock {\it Nonlinearity} {\bf 35}, 1702 (2022).

\bibitem{pos22b}
C.~M.~Postlethwaite, R.~Sturman.
Stability of heteroclinic cycles in ring graphs.
\newblock {\it Chaos} {\bf 32}, 063104 (2022).

\bibitem{pj98}
M.R.E.~Proctor and C.A.~Jones.
\newblock The interaction of two spatially resonant patterns in thermal
convection. I. Exact 1:2 resonance.
\newblock {\it J. Fluid. Mech.} {\bf 188}, 301 -- 335 (1998).

\bibitem{rab06}
M.I.~Rabinovich, P.~Varona, A.~Selverston, H.D.I.~Abarbanel.
Dynamical principles in neuroscience.
{\it Rev. Mod. Phys.} {\bf 78}, 1213--1265 (2006).

\end{thebibliography}
\end{document}